\documentclass[11pt,letterpaper]{amsart}
\usepackage[utf8]{inputenc}
\usepackage{amsmath,amssymb,graphicx,setspace,verbatim,url}
\usepackage{color}
\usepackage{hyperref}
\usepackage{multicol}
\usepackage{enumitem}
\usepackage{float}
\usepackage{soul}
\usepackage{tikz}
\usepackage[normalem]{ulem}
\usepackage[margin=3cm]{geometry}

\usepackage{tikz-cd}
\usetikzlibrary{arrows.meta}
\newenvironment{blue}{\relax\color{blue}}{\hspace*{.5ex}\relax}
\newcommand{\bcp}{\begin{blue} }
\newcommand{\ecp}{\end{blue}}
\usepackage{listings}
\usepackage{comment}

\tikzcdset{
  arrow style={tikz,diagrams={>=Triangle}}
}

\input xy
\xyoption{all}

\newtheorem{prop}{Proposition}[section]
\newtheorem{coro}[prop]{Corollary}
\newtheorem{thm}[prop]{Theorem}
\newtheorem{lemma}[prop]{Lemma}

\newtheorem{definition}[prop]{Definition}

\newtheorem{example}[prop]{Example}

\DeclareMathOperator{\Cor}{Cor}
\DeclareMathOperator{\Out}{Out}

\newcommand{\TT}{\mathcal{T}}

\newcommand{\Aut}{\mathrm{Aut}}

\newcommand{\JJ}{[J_\alpha]}

\newtheorem*{continuancex}{Example \continuanceref\ (cont.)}

\makeatletter
\newcommand{\getnamereftext}[1]{%
  \@ifundefined{r@#1}{}{%
    \unexpanded\expandafter\expandafter\expandafter{%
      \expandafter\expandafter\expandafter\@thirdoffive\csname r@#1\endcsname
    }%
  }%
}
\makeatletter

\newenvironment{continuance}[1]{%
  \newcommand\continuanceref{\ref{#1}}%
  \if\relax\getnamereftext{#1}\relax
    \continuancex
  \else
    \continuancex[\nameref{#1}]%
  \fi
}{\endcontinuancex}

\begin{document}
\title[The geometry of wreath and semi-direct products]{The geometry of wreath and semi-direct products}

\author[Claudio Alexandre Piedade]{Claudio Alexandre Piedade}
\address{Claudio Alexandre Piedade, Universit\'e Libre de Bruxelles, D\'epartement de Math\'ematique, C.P.216 - Alg\`ebre et Combinatoire, Boulevard du Triomphe, 1050 Brussels, Belgium, Orcid number 0000-0002-0746-5893
}
\email{claudio.piedade@ulb.be}

\author[Philippe Tranchida]{Philippe Tranchida}
\address{Philippe Tranchida, Universit\'e Libre de Bruxelles, D\'epartement de Math\'ematique, C.P.216 - Alg\`ebre et Combinatoire, Boulevard du Triomphe, 1050 Brussels, Belgium, Orcid number 0000-0003-0744-4934.}
\curraddr{}
\email{tranchida.philippe@gmail.com}
\urladdr{}

{}

\setcounter{tocdepth}{1}

\date{\today}

\begin{abstract}
Coset geometries are incidence geometries constructed from a group $G$ and a system of subgroups $(G_i)_{i \in I}$ of subgroups of $G$. For any algebraic group operation, it is then natural to wonder whether it can be extended to the framework of coset geometries. This has been achieved in the case of the halving (\cite{halving}) and in the case of free (amalgamated) products, HNN-extensions, and semi-direct products (\cite{piedade2025group}). In this article, we explore more deeply two operations related to semi-direct products: the twisting and the wreath product. We show that these operations extend to coset geometries in such a way that they preserve key properties, such as flag-transitivity, residual-connectedness and being thin. In particular, we can apply twistings and wreath products to polytopes and hypertopes. Doing so, we show that there exists regular polytopes and hypertopes for almost-simple group with socle a sporadic simple group.

\end{abstract}

\maketitle


\section{Introduction}

Coset incidence systems are incidence systems (see Section \ref{sec:prelims} for precise definitions) that are constructed from a group $G$ together with a system $(G_i)_{i\in I}$ of subgroups of $G$. These structures, introduced by Jacques Tits~\cite{Tits1957,Tits1963geometries}, are the basis for a variety of combinatorial and geometric objects, such as abstract polytopes, vector and projective spaces, and Tits' buildings~\cite{Handbook,Tits1974,ARP}. In particular, the theory of buildings has had profound consequences across mathematics, illustrating how interesting combinatorial structures can emerge from algebraic objects~\cite{Handbook,KenBrownBook2}. Regarding abstract polytopes, coset incidence systems, and in particular the Tits' algorithm~\cite{Tits1957}, are the key to proving the one-to-one correspondence between abstract regular polytopes and the class of groups known as string $C$-groups.

For any operation that can be performed on groups, it is natural to try to extend it in a coherent way to systems $(G_i)_{i\in I}$ of subgroups of $G$ so that the operation can then be applied in the context of coset incidence systems.
Whenever that is possible, it is also interesting to try to give a geometrical interpretation to these algebraic constructions.
Our goal is to develop this concept, where algebraic operations are the inspiration for defining new geometric operations.

This approach has already been implemented successfully in \cite{halving}, where the group operation under consideration is the halving (see ~\cite[Chapter 7]{ARP}).
Indeed, the halving is classically defined only for rank 3 polytopes and only as a group operation. In~\cite{halving}, we generalized this algebraic operation in a natural way to a much larger class of objects, including the non-degenerate leaf-hypertopes. Additionally, this operation was given a concrete geometrical construction, based on previously described operations studied in~\cite{LefvrePercsy2000}.
Additionally, in \cite{piedade2025group}, we pursued this philosophy for a broader scope of operations on groups. There, we covered the cases of free products, free-amalgamated products, HNN-extensions, semi-direct products and a restricted version of twisting. These constructions were shown to have concrete applications, particularly for the study and classification of Shephard groups, a generalization of both Coxeter and Artin-Tits groups.

In this paper, we build upon our previous results and give an extended and generalized definition of twisting of coset incidence systems based on the one introduced in \cite{piedade2025group}. This generalized version allows us not only to apply this operation to a broader range of new examples, but also to define a wreath product of coset incidence systems. The results of this paper expand the range of tools available for a systemic study of coset incidence geometries and the way these can be combined, offering new ways to tackle a variety of distinct problems.

In particular, many families of potential examples of high rank regular polytopes can be described as wreath products. In order to prove these families are actually regular polytopes, one must prove an intersection property of their maximal parabolic subgroups. With our new machinery for wreath products, we are able to handle many of these cases at once, partially solving a question posed in~\cite{FernandesPiedade2026}.

Moreover, using the generalized twisting operation, we are able to prove the existence of abstract regular polytopes and regular hypertopes with non-linear diagrams for almost simple groups of sporadic type $M_{12}$,
$M_{22}$, $J_2$, $J_3$, $McL$, $HS$, $He$, $Suz$, and $ON$.
This is mostly achieved by applying the twisting construction to the setting of self-dual polytopes whose automorphism group is a sporadic simple group.
These results constitute a step forward towards a full answer to a question posed by Nuzhin in 2026's edition of Kourovka Notebook:~\cite{KourovkaNotebook}
\begin{quote}
    ``Which finite almost simple groups are the automorphism groups of regular polytopes of rank 3?"
\end{quote}
Similar results are already known for almost simple group with socle $PSL(2,q)$~\cite{Connor2015}, $PSL(3,q)$~\cite{Brooksbank2009},
$Alt(n)$~\cite{Fernandes2011,Fernandes2012}, 
$Sz(q)$~\cite{Leemans2006Almost}, and $R(q)$~\cite{Leemans2017}.

We conclude the paper with another use of the wreath product of coset incidence systems to construct an interesting rank two geometry for the (extended) lamplighter group, one of the most famous examples of wreath products. We light-heartedly call this geometry the Schr\" odinger street geometry, as it can be interpreted as some sort of quantum version of the classical lamplighter street on which the lamplighter group acts on.

\subsection{Acknowledgments}
The first author is funded by an Action de Recherche Concert\'{e}e -- ARC -- from the Communaut\'{e}e Fran\c{c}aise Wallonie-Bruxelles and the second author is a Postdoctoral Researcher of the Fonds de la Recherche Scientifique -- FNRS. We thank Dimitri Leemans for useful discussion regarding almost simple groups.

\section{Background and Preliminary results}\label{sec:prelims}
In this section, we regroup all the necessary definitions and well-established results that will be used later on. We start by generalities on incidence geometries and coset geometries and then give some context on (string) $C$-groups and permutation representation graphs. 
Notice that all index sets $I$ will be supposed to be finite, unless they are explicitly said to be infinite.

\subsection{Incidence geometry and Coset Geometry}
\label{sec:back:subsec:incidence}
    A 4-tuple $\Gamma = (X,I,*,t)$ is called an \textit{incidence system} if
    \begin{enumerate}
        \item $X$ is a set whose elements are called the \textit{elements} of $\Gamma$,
        \item $I$ is the set of types of $\Gamma$, 
        \item $*$ is a symmetric and reflexive relation (called the \textit{incidence relation}) on $X$, and
        \item $t$ is a map from $X$ to $I$, called the \textit{type map} of $\Gamma$, such that distinct elements $x,y \in X$ with $x * y$ satisfy $t(x) \neq t(y)$. 
    \end{enumerate}

The \textit{rank} of $\Gamma$ is the cardinality of the type set $I$.
In an incidence system $\Gamma$, a \textit{flag} is a set of pairwise incident elements. The rank of a flag, \text{rank}($F$), is the cardinality of $F$, that is $\textnormal{rank}(F)=|F|$. The type of a flag $F$, $t(F)$, is the set of types of the elements of $F$. 
A \textit{chamber} is a flag of type $I$. An incidence system $\Gamma$ is an \textit{incidence geometry} if all its maximal flags are chambers.

The \textit{incidence graph} of $\Gamma$ is a graph with vertex set $X$ and where two elements $x$ and $y$ are connected by an edge if and only if $x * y$.
We say $\Gamma$ is \textit{finite} if the incidence graph of $\Gamma$ is finite.

Let $F$ be a flag of $\Gamma$. An element $x\in X$ is {\em incident} to $F$ if $x*y$ for all $y\in F$. The \textit{residue} of $\Gamma$ with respect to $F$, denoted by $\Gamma_F$, is the incidence system formed by all the elements of $\Gamma$ incident to $F$ but not in $F$. Whenever $\Gamma$ is a geometry, the \textit{rank} of the residue $\Gamma_F$ is defined to be rank$(\Gamma) - |F|$. If $\Gamma$ is of rank $n$, then the \textit{corank} of a flag $F$ is the difference between the rank of $\Gamma$ and the rank of $F$.

An incidence geometry $\Gamma$ is \textit{connected} if its incidence graph is connected. It is \textit{residually connected} if all its residues of rank at least two are connected. Note that, by \cite[Lemma 1.6.3]{buekenhout2013diagram}, a residually connected geometry $\Gamma$ of rank greater or equal to $2$ is always connected. On the other hand, rank one geometries are all trivially residually connected but are never connected.

An incidence geometry is \textit{firm}, \textit{thin} or \textit{thick} if all its residues of flags of corank one contain, respectively, at least two elements,  exactly two elements, more than two elements.

Let $\Gamma = \Gamma(X,I,*,t)$ be an incidence geometry. A {\em correlation} of $\Gamma$ is a bijection $\phi$ of $X$ respecting the incidence relation $*$ and such that, for every $x,y \in X$, if $t(x) = t(y)$ then $t(\phi(x)) = t(\phi(y))$. If, moreover, $\phi$ fixes the types of every element (i.e $t(\phi(x)) = t(x)$ for all $x \in X$), then $\phi$ is said to be an {\em automorphism} of $\Gamma$. The group of all correlations of $\Gamma$ is denoted by $\Cor(\Gamma)$ and the automorphism group of $\Gamma$ is denoted by $\Aut(\Gamma)$. Remark that $\Aut(\Gamma)$ is a normal subgroup of $\Cor(\Gamma)$ since it is the kernel of the action of $\Cor(\Gamma)$ on $I$.

If $\Aut(\Gamma)$ is transitive on the set of flags of $\Gamma$ then we say that $\Aut(\Gamma)$ is {\em flag-transitive} on $\Gamma$. We say $\Aut(\Gamma)$ is {\em chamber-transitive} on $\Gamma$ if $\Aut(\Gamma)$ is transitive on the set of chambers of $\Gamma$. Note that, if $\Gamma$ is a geometry, being chamber-transitive is equivalent to being flag-transitive~\cite[Proposition 2.2]{hypertopes}. If moreover, the stabilizer of a chamber in $\Aut(\Gamma)$ is reduced to the identity, we say that $\Gamma$ is {\em simply transitive} or {\em regular}. If $G$ is any group together with a homomorphism $\psi \colon G \to \Aut(\Gamma)$, we also say that $G$ is flag-transitive, simply transitive or regular on $\Gamma$ if that is true of $\psi(G) \leq \Aut(\Gamma)$.
A \textit{regular hypertope} is a thin, residually connected and flag-transitive incidence geometry~\cite{hypertopes}.

The main objects of interest in this paper are incidence geometries that are obtained from a group $G$ together with a set $(G_i)_{i \in I}$ of subgroups of $G$ as described in~\cite{Tits1957}. 
    A \emph{coset incidence system} $\Gamma(G,(G_i)_{i \in I})$, denoted also as $(G,(G_i)_{i \in I})$, is the incidence system over the type set $I$ where:
    \begin{enumerate}
        \item The elements of type $i \in I$ are left cosets of the form $g G_i $, $g \in G$.
        \item The incidence relation is given by non-empty intersection. More precisely, the element $g G_i$ is incident to the element $k G_j $ if and only if $g G_i \cap k G_j \neq \emptyset$.
    \end{enumerate}
We usually talk about the subgroups $G_i$ as the \emph{(standard) maximal parabolic subgroups}. For $J\subseteq I$, we define \emph{(standard) parabolic subgroups} as $G_J=\cap_{i\in J} G_i$, and define $G_\emptyset=G$.
When $J=I\setminus\{i\}$, we say that $G_J=G_{I\setminus\{i\}}$ are the \emph{minimal parabolic subgroups}, and we will denote them as $G^i$. Finally, the subgroup $G_I=\cap_{i\in I} G_i$ is called the \emph{Borel subgroup} of $\Gamma$. 
    
A coset incidence system that is an incidence geometry is called a \emph{coset geometry}.
Flag-transitive geometries can always be constructed as coset geometries.
The following theorem gives a group theoretical condition to check the flag-transitivity of $G$ on the coset incidence system $\Gamma(G,(G_i)_{i\in I})$. Notice that $G$ always naturally acts on $\Gamma(G,(G_i)_{i \in I})$ by left multiplication.

\begin{thm}\cite[Theorem 1.8.10]{buekenhout2013diagram}\label{thm:cosetFT}
Let $\Gamma = (G,(G_i)_{i\in I})$ be a coset incidence system. Then $G$ is flag-transitive on $\Gamma$ if and only if $G_JG_i = \cap_{j \in J}(G_jG_i)$ for each $J \subseteq I$ and each $i\in I\setminus J$.
Additionally, if $G$ is flag-transitive on $\Gamma$, then $\Gamma$ is a geometry.
\end{thm}

The group-theoretical condition for flag-transitivity can be stated in many different ways, as shown in the following proposition.

\begin{prop}\label{prop:FTequivs} \cite[Proposition 2.2]{piedade2025group}
    Let $\Gamma = (G,(G_i)_{i\in I})$ be a coset incidence system. Then, the following statements are equivalent:
    \begin{enumerate}
        \item $G$ is flag-transitive on $\Gamma$;\label{prop:FTequivs:itm:IsFT}
        \item for each $J\subseteq I$ and each $i\in I\setminus J$, $G_JG_i=\cap_{j\in J}(G_jG_i)$;\label{prop:FTequivs:itm:FTBueken}
        \item for each $J\subseteq 
        I$, there is a natural isomorphism of incidence systems over $I\setminus J$
        $$\phi_J: \Gamma(G_J, (G_{J\cup \{i\}})_{i\in I\setminus J})\to \Gamma_{\{G_j\mid j\in J\}}$$
        given by $\phi_J(aG_{J\cup \{i\}})= a G_i$, where $a\in G_J$ and $i\in I\setminus J$;\label{prop:FTequivs:itm:Isomorphism}
        \item for each distinct types $i,k\in I$, each subset $J\subseteq I\setminus\{i,k\}$ and each element $g\in G_J$, if $G_i\cap gG_k\neq \emptyset$, then $G_J\cap G_i\cap gG_k \neq \emptyset$;\label{prop:FTequivs:itm:FTPassini}
        \item for each non-empty subset $J\subseteq I$, if there is a family of elements of $G$, $(g_j)_{j\in J}$,  such that, for every choice of $i,j\in J$, $g_iG_i \cap g_j G_j\neq \emptyset$, then $\cap _{j\in J}g_jG_j = g G_J$ for some $g\in G$;\label{prop:FTequivs:itm:FTPassiniProp}
        \item for every choice of subsets $J,H,K$ of $I$ and of elements $f,g,h$ of $G$, if the cosets $fG_J$, $gG_H$ and $hG_K$ have pairwise nonempty intersection, then $fG_J\cap gG_H\cap hG_K\neq \emptyset$;\label{prop:FTequivs:itm:FTIntersectionCosets}
        \item for every $J,H,K\subseteq I$ we have 
            $(G_J\cap G_H)(G_J\cap G_K)= G_J\cap (G_HG_K)$;\label{prop:FTequivs:itm:FTProductOfIntersections}
        \item for every $J,H,K\subseteq I$ we have 
        $(G_JG_H)\cap (G_JG_K) = G_J(G_H\cap G_K)$.\label{prop:FTequivs:itm:FTIntersectionOfProducts}
    \end{enumerate}
\end{prop}

Residual connectedness can also be checked using group theoretical condition, under the assumption of flag-transitivity.

\begin{lemma}\cite[Proposition 1.8.12]{buekenhout2013diagram} \label{prop:RCcondi}
    Let $\Gamma = (G,(G_i)_{i\in I})$ be a coset incidence system. Then the following statements are equivalent
    \begin{equation}\tag{RC1}
        G_J = \langle G_{\{i\} \cup J} | i \in I \setminus J \rangle \textnormal{ for all }J \subset I \textnormal{ such that }|I \setminus J| \geq 2.
    \end{equation}
    \begin{equation}\tag{RC2}
        G_J = \langle G_{\{i\} \cup J} , G_{\{k\} \cup J} \rangle \textnormal{ for all }J \subset I \textnormal{ and distinct }i,k\in I\setminus J.
    \end{equation}
Moreover, if $G$ acts flag-transitively on $\Gamma$, then $\Gamma$ is residually connected if and only if one (and thus both) of the the above condition holds.
\end{lemma}

The \textit{minimal parabolic subgroups} of $\Gamma = (G,(G_i)_{i\in I})$ are the groups $G^i :=G_{I\setminus\{i\}}$. For a subset $J\subseteq I$, we let $G^J=\langle G^i\mid i\in J\rangle$. In other words, $G^J$ is the subgroup of $G$ generated by the minimal parabolic subgroups $G^i$, for $i\in J$. Notice that $G^I$ does not need to be necessarily equal to $G$. With these notations, the following conditions are all equivalent to (RC1) and (RC2).

\begin{lemma}\label{lem:RCequivs}\cite[Lemma 2.5]{piedade2025group}
    Let $\Gamma = (G,(G_i)_{i\in I})$ be a coset incidence system. Then the following statements are equivalent:
    \begin{enumerate}
        \item for all $J \subset I$ such that $|I \setminus J| \geq 2$, we have
        $G_J = \langle G_{\{i\} \cup J} | i \in I \setminus J \rangle$\label{lem:RCequivs:itm:RC1}
        \item for all $J \subset I$ and distinct $i,k\in I\setminus J$, we have 
        $G_J = \langle G_{\{i\} \cup J} , G_{\{k\} \cup J} \rangle$ \label{lem:RCequivs:itm:RC2}
        \item for every $J\subseteq I$, $G_{J} = G^{I\setminus J}$.\label{lem:RCequivs:itm:BottomUp}
        \item for every $J\subseteq I$, $G^J = G_{I\setminus J}$.\label{lem:RCequivs:itm:UpBottom}
        \item for every $J,K\subseteq I$, $G^J\cap G^K = G^{J\cap K}$.\label{lem:RCequivs:itm:Intersection}
        \item for all $J \subset I$ with $| I \setminus J| \geq 2$ and for any system $(G_{K_i})$ of parabolic subgroups such that $I \setminus J \subseteq \cup_i (I \setminus K_i)$, we have that $G_J \subseteq \langle G_{K_i} \rangle$. Note that the equation $I \setminus J \subseteq \cup_i (I \setminus K_i)$ can be rewritten as $\cap_i K_i \subseteq J$.\label{lem:RCequivs:itm:Parabolics}
    \end{enumerate}
    Moreover, if one (hence, all) of this statements is true, and $\Gamma$ is flag-transitive, then $\Gamma$ is a residually connected coset geometry.
\end{lemma}

Finally, we recall the group theoretical conditions to verify whether a coset incidence system is firm, thin or finite.

\begin{thm}\label{thm:CosetFirmThinFinite}\cite[Corollary 1.8.15]{buekenhout2013diagram}
     Suppose that $\Gamma = (G,(G_i)_{i\in I})$ is a coset incidence system over
the finite set $I$ on which G acts flag-transitively. Then,

\begin{enumerate}
    \item $\Gamma$ is a firm geometry if and only if
        \begin{equation*}\tag{FIRM}
            G_{I\setminus \{j\} } \neq G_I \textnormal{ for each }j \in I.
        \end{equation*}
    \item $\Gamma$ is a thin geometry if and only if
        \begin{equation*}\tag{THIN}
            [G_{I\setminus \{j\} } : G_I] = 2 \textnormal{ for each }j \in I.
        \end{equation*}
    \item $\Gamma$ is finite if and only if $[G:G_I]$ is finite.
\end{enumerate}
\end{thm}

Coset incidence systems that satisfy conditions (a) or (b) in the above statement will be said to satisfy (FIRM) or (THIN) accordingly. This terminology is useful to talk about this property without assuming flag-transitivity.
\subsection{String $C$-groups and their permutation representation graphs}\label{subsec:StringCGroups}

A \emph{$C$-group} is a pair $(G,S)$ where $G$ is a group and $S = \{\rho_i\mid i\in I\}$ is a finite set of involutions generating $G$  that satisfy the following property, called the \emph{intersection property}.
$$\forall M,N\subseteq S, \langle M\rangle\cap \langle N\rangle = \langle M\cap N\rangle.$$
A $C$-group is said to be of \emph{rank $r$} if $|S|=r$. It is common to use the index set $I = \{0,\ldots,r-1\}$.

We say a $C$-group is \emph{string} if $S$ is an ordered set such that it satisfies the following property, called the \emph{commuting property} or \emph{string property}.
$$\forall i,j\in\{0,\ldots, r-1\}, \;|i-j|>1\Rightarrow (\rho_i\rho_j)^2=1.$$
The \emph{dual} of a string $C$-group is obtained by reversing the ordering of the generators.

The \emph{Coxeter diagram} of a $C$-group $(G,S)$, denoted $D(G,S)$, is the graph whose nodes corresponds to elements of $S$ and where the edge between the generators $\rho_i$ and $\rho_j$ has label $p_{i,j}:=o(\rho_i\rho_j)$, the order of $\rho_i\rho_j$. By convention, edges with label equal to 2 are not drawn, and, whenever an edge should be labelled with 3, its label is omitted. For string $C$-groups, the commuting property forces the diagram to be linear. It must then be as in Figure~\ref{coxeterdiagram}.
Note that Coxeter groups are themselves $C$-group where the only relations imposed on the group are those described by the diagram. In other words, every $C$-group is a quotient of a Coxeter group.
\begin{figure}[h]
 $$\xymatrix@-1.9pc{
*{\bullet} \ar@{-}[rrrrr]^{p_{0,1}} && & &&*{\bullet} \ar@{-}[rrrrr]^{p_{1,2}} && & &&*{\bullet} \ar@{--}[rrrrrr] && && && \ar@{--}[rrrr] && && *{\bullet} \ar@{-}[rrrrr]^{p_{j-1,j}} && & &&  *{\bullet} \ar@{--}[rrrrrr] && && && \ar@{--}[rrrr] && && *{\bullet} \ar@{-}[rrrrr]^{p_{n-2,n-1}} && & && *{\bullet} \\
*{\rho_0} && & &&*{\rho_1} && & && *{\rho_2} && && && && && *{\rho_{j-1}} && & && *{\rho_{j}} && && && && && *{\rho_{n-2}} && & && *{\rho_{n-1}}\\
}$$
 \caption{Coxeter Diagram of a string $C$-group.}
 \label{coxeterdiagram}
\end{figure}
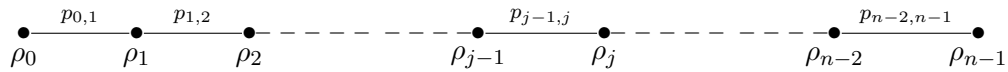

Suppose that $G$ is a permutation group of degree $n$ and let $(G, \{\rho_0,\ldots,\rho_{r-1}\})$ be a $C$-group. The \emph{permutation representation graph} $\mathcal{G}$ of $(G,S)$ is an $r$-edge-labelled multi-graph with $n$ vertices and with an $i$-edge
 $\{a,\,b\}$ whenever $a\rho_i=b$ with $a\neq b$ and $i\in\{0,\ldots,r-1\}$ (for more details, we refer to~\cite{FernandesPiedade2026,Fernandes2012}).
 
String $C$-groups are important groups in the study of regular polytopes; indeed, string $C$-groups are in one-to-one correspondence with abstract regular polytopes. From a string $C$-group $(G,S)$, one can define its maximal parabolic subgroups as $G_i:=\langle S\setminus\{\rho_i\}\rangle$, for every $i\in I$, and build a coset incidence system $(G,(G_i)_{i\in I})$ that is a regular polytope, or in other words, a regular hypertope with linear Buekenhout diagram (see~\cite{buekenhout2013diagram,hypertopes} for more details).
On the other hand, the automorphism group of a regular hypertope is always a $C$-group. Hence, studying (string) $C$-groups give information regarding the geometrical structure these encode.

\section{Generalizing Twisting}

In this section, we will introduce two constructions: the direct product construction (Section~\ref{sec:DirectProduct}) and the generalized twisting construction (Section~\ref{sec:GenralizedTwisting}).

The direct product construction is, to some extend, a well known construction in the folklore of incidence geometry. Note that it can be seen also as a subcase of the generalized twisting construction. We start with the direct product case here for two reasons. First, it serves as a gentle introduction to notations that will be used in further sections. Secondly, we show how to extend this construction to direct products and direct sums of infinitely many coset incidence systems. This will be useful at the end of the article, when we discuss incidence geometries for the lamplighter group.

Regarding the generalized twisting construction, we extend a previous definition of this operation introduced in~\cite{piedade2025group}. In the definition of this operation given in~\cite{piedade2025group}, the action of the group $B$ was quite restricted. Thanks to our new definition introduce here, we are then able to define the wreath product construction given in Section~\ref{sec:Wreath}, which, as we will demonstrate, is very useful in practice.

\subsection{Direct Product}\label{sec:DirectProduct}

We start by giving a construction for the direct product of two coset incidence systems.  

Let $\alpha = (A,(A_i)_{i \in I_\alpha})$ and $\beta = (B,(B_i)_{i \in I_\beta})$ be two coset incidence systems.

\begin{definition}
    The direct product of $\alpha$ and $\beta$ is the coset geometry $\alpha \times \beta = (G, (G_i)_{i \in I}$) where $I = I_\alpha \sqcup I_\beta$, $G = A \times B$ and  
    \begin{equation}
 G_i= \begin{cases}
     A \times B_i, & \text{if $i \in I_\beta$}.\\
     A_i \times B, & \text{if $i \in I_\alpha$}.
  \end{cases}
\end{equation}
\end{definition}

This construction can be realized directly at the level of the incidence systems, without referring to groups. Indeed, if $\alpha$ and $\beta$ are two incidence system, one can define $\alpha \times \beta$ to be the incidence system where elements are the disjoint union of the elements of $\alpha$ and $\beta$ and where every element of $\alpha$ is incident to every element of $\beta$. 

We now state without proof two results concerning $\alpha \times \beta$. These are easy to prove directly, but can also be seen as a corollary of Theorem~\ref{thm:GenTwisting_FT_Geo_RC_FRM}. Indeed, the direct product of coset incidence system is a special case of the semi-direct product, where the action of $B$ on $A$ is trivial, and in that case $\alpha$ is always $(\beta,\varphi)$-admissible (see Section \ref{sec:GenralizedTwisting}).

\begin{lemma}\label{lem:directproduct_parabolics}
    For any $J \subset I = I_\alpha \sqcup I_\beta$, we have
    $$ G_J = A_{J_\alpha} \times B_{J_\beta}$$
    where $J_\alpha = J \cap I_\alpha$ and $J_\beta = J \cap I_\beta$.
\end{lemma}

\begin{thm}\label{thm:AlphaTimesBeta}
   Let $\alpha=(A,(A_i)_{i\in I_\alpha})$ and $\beta=(B,(B_i)_{i\in I_\beta})$ be two coset incidence systems with $I_\alpha$ and $I_\beta$ both finite. Suppose that $A$ and $B$ are flag-transitive on $\alpha$ and $\beta$, respectively. Then
    \begin{enumerate}
        \item the group $G=A\times B$ is flag-transitive in $\alpha\times\beta$;
        \item $\alpha\times \beta$ is a coset geometry;
        \item $\alpha\times \beta$ is a finite coset geometry if and only if $\alpha$ and $\beta$ are both finite coset geometries;
        \item $\alpha\times \beta$ is residually connected if both $\alpha$ and $\beta$ are residually connected;
        \item $\alpha \times \beta$ is firm (resp. thin) if both $\alpha$ and $\beta$ are firm (resp. thin).
    \end{enumerate} 
\end{thm}

The direct product construction can be iterated, and in particular, we can define $\alpha_n := \underbrace{\alpha\times \alpha \times \cdots\times \alpha}_{n\textnormal{ factors}}$. 

In order to go further and tackle the case of infinitely many coset incidence systems, it seems that we have a choice between direct products and direct sums. It turns out that for coset geometries, both choice lead to the same underlying incidence geometry. To see this, let us construct the geometry directly, without using groups.

Let $\Omega$ be any index set (possibly infinite) and let $(\Gamma_\omega)_{\omega \in \Omega}$ be a family of incidence systems indexed by $\Omega$. More precisely, we have $\Gamma_\omega = (X_\omega, \star_\omega, \tau_\omega)$ over the type set $I_\omega$.

\begin{definition}
    The \textit{join} of $(\Gamma_\omega)_{\omega \in \Omega}$ is the incidence system $\bigsqcup_{\omega\in \Omega} \Gamma_\omega = (X , \star, \tau)$ over $I = \bigsqcup_{\omega \in \Omega} I_\omega$ where
    \begin{itemize}
        \item $X = \bigsqcup_{\omega \in \Omega} X_\omega$,
        \item $\tau(x) = \tau_\omega(x)$ where $\omega$ is the unique index element such that $x \in X_\omega$,
        \item $x \star y$ if and only if either $x$ and $y$ belong to distinct $X_\omega$ or $x,y \in X_\omega$ and $x \star_\omega y$.
    \end{itemize}
\end{definition}

Suppose now that each $\Gamma_\omega$ is expressed as a coset incidence system $(^\omega G, (^\omega G_i)_{i \in I_\omega})$. The join $\bigsqcup_{\omega\in \Omega} \Gamma_\omega$ can then be realized by using direct products or direct sums. 

\begin{definition}
    The direct sum of a family $\Gamma_\omega =(^\omega G, (^\omega G_i)_{i \in I_\omega})$ of coset incidence systems is the coset incidence system $\bigoplus_{\omega \in \Omega} \Gamma_\omega = (G, (G_i)_{i \in I})$ where $I = \bigsqcup_{\omega \in \Omega} I_\omega$  and
    \begin{itemize}
        \item $G =  \bigoplus_{\omega \in \Omega} {^\omega} G$,
        \item Fix the unique $\lambda$ such that $i \in I_\lambda$. We define $G_i$ to be the subgroup of $G$ where the $\lambda$ factor is restricted to elements of $^\lambda G_i$. Equivalently, $G_i = \bigoplus_{\omega \in \Omega} {^\omega} H$ where $^\omega H = {^\omega G}$ if $\omega \neq \lambda$ and $^\lambda H = {^\lambda G} _i$  
    \end{itemize}

    Similarly, the direct product $\prod_{\omega \in \Omega} \Gamma_\omega $ is the coset incidence system $\Gamma_\omega = (G, (G_i)_{i \in I})$ where $I = \bigsqcup_{\omega \in \Omega} I_\omega$  and
    \begin{itemize}
        \item $G =  \prod_{\omega \in \Omega} {^\omega} G$,
        \item Fix the unique $\lambda$ such that $i \in I_\lambda$. We define $G_i$ to be the subgroup of $G$ where the $\lambda$ factor is restricted to elements of $^\lambda G_i$. Equivalently, $G_i = \prod_{\omega \in \Omega} {^\omega} H$ where $^\omega H = {^\omega G}$ if $\omega \neq \lambda$ and $^\lambda H = {^\lambda G} _i$  
    \end{itemize}
\end{definition}

\begin{thm}\label{thm:join}
    Let $\Gamma_\omega = (^\omega G, (^\omega G_i)_{i \in I_\omega})$ be a family of coset incidence systems indexed by a set $\Omega$ and such that $I_\omega$ is finite for all $\omega \in \Omega$. Then, both $\bigoplus_{\omega \in \Omega} \Gamma_\omega$ and $\prod_{\omega \in \Omega} \Gamma_\omega$ are isomorphic to the join $\bigsqcup_{\omega\in \Omega} \Gamma_\omega$.
\end{thm}
\begin{proof}
    Let $i \in I$ and let $\omega$ be the unique index element such that $i \in I_\omega$. In all three incidence systems, the elements of type $i$ are naturally identified with cosets of $^\omega G_i$ in $^\omega G$. Under these identification, the incidence in both $\bigoplus_{\omega \in \Omega} \Gamma_\omega$ and $\prod_{\omega \in \Omega} \Gamma_\omega$ is also easily verified to be the one of $\bigsqcup_{\omega\in \Omega} \Gamma_\omega$.
\end{proof}

It is important to note that, while it is always possible to construct the join $\bigsqcup_{\omega\in \Omega} \Gamma_\omega$ via the direct sum $\bigoplus_{\omega \in \Omega} {^\omega G}$, the group $\bigoplus_{\omega \in \Omega} {^\omega G}$ never acts flag-transitively on $\bigsqcup_{\omega\in \Omega} \Gamma_\omega$ whenever $\Omega$ is infinite. Indeed, two flags of the same type $J \subset I$ in $\bigsqcup_{\omega\in \Omega} \Gamma_\omega$ may differ in infinitely many elements, and there is no element of $\bigoplus_{\omega \in \Omega} {^\omega G}$ sending one to the other. That being said, the group $\bigoplus_{\omega \in \Omega} {^\omega G}$ does act transitively on flags of $\bigsqcup_{\omega\in \Omega} \Gamma_\omega$ that differ in only finitely many elements, and thus in particular on flags of finite rank.

On the other hand, we have that $\bigsqcup_{\omega\in \Omega} \Gamma_\omega$ is residually-connected whenever all $\Gamma_\omega$ are residually connected and that, if $^\omega G$ acts transitively on $\Gamma_\omega$ for all $\omega \in \Omega$, then $\prod_{\omega \in \Omega} {^\omega G}$ acts transitively on $\bigsqcup_{\omega\in \Omega} \Gamma_\omega$. 

\begin{thm}\label{thm:AlphaTimesBetaInfinite}
  Let $\Gamma_\omega = (^\omega G, (^\omega G_i)_{i \in I_\omega})$ be a family of coset incidence systems indexed by a set $\Omega$ and such that $I_\omega$ is finite for all $\omega \in \Omega$. Suppose that, for each $\omega \in \Omega$, the group $^\omega G$ acts flag-transitively on $\Gamma_\omega$. Then
    \begin{enumerate}
        \item the group $G = \prod_{\omega \in \Omega} {^\omega} G$ acts flag-transitively on $\prod_{\omega \in \Omega} \Gamma_\omega$;
        \item $\prod_{\omega \in \Omega} \Gamma_\omega$ is a coset geometry;
        \item $\prod_{\omega \in \Omega} \Gamma_\omega$ is residually connected if all $\Gamma_\omega$ are residually connected;
        \item $\prod_{\omega \in \Omega} \Gamma_\omega$ is firm (resp. thin) if all $\Gamma_\omega$ are firm (resp. thin).
    \end{enumerate} 
\end{thm}
\begin{proof}
    The statement of the theorem is easily seen to hold for the join incidence system $\bigsqcup_{\omega\in \Omega} \Gamma_\omega = (X , \star, \tau)$, with ${^\omega} G$ replaced by $\Aut(\Gamma_\omega)$ . We can then conclude by using Theorem \ref{thm:join}.
\end{proof}

\subsection{Generalized Twisting}\label{sec:GenralizedTwisting}

We now introduce the generalized twisting of two coset incidence systems. 



Let $\alpha=(A,(A_i)_{i\in I_\alpha})$ and $\beta=(B,(B_i)_{i\in I_\beta})$ be two coset incidence systems, and suppose that we have an action $\varphi$ of $B$ on $A$ by automorphisms.

\begin{definition}\label{def:GammaPhiAdmiss}
    We say that $\alpha$ is $(\beta, \varphi)$-admissible 
    if
    \begin{enumerate}
        \item The action $\varphi$ of $B$ on $A$ permutes the maximal parabolic subgroups $(A_i)_{i\in I_\alpha}$;
        \item Consider that $\varphi$ is also acting on the type set $I_\alpha$ in the natural way and let $K$ to be the set of orbits of $B$ on $I_\alpha$.
        For each orbit of type $L\in K$, there exists $F_L\in L$, such that for all $M,N\subseteq I_\beta$, we have that
\begin{equation}\tag{IPO}\label{eq:IPO}
          ^LO_M\cap\ ^LO_N = \ ^LO_{M\cap N} 
        \end{equation}
        where $^LO_M$ is the orbit of $F_L$ under the action of the parabolic subgroup $B_{I_\beta\setminus M}$. 
    \end{enumerate}
\end{definition}

Here is a typical example to which our setting applies. We recommend following this example throughout this section.

\begin{example}\label{example1}
    Let $A=\langle a_i\mid i\in \{0,\ldots,6\}\rangle$ and $B=\langle b_7,b_8\rangle$ be the Coxeter groups with the following Coxeter diagrams.
$$
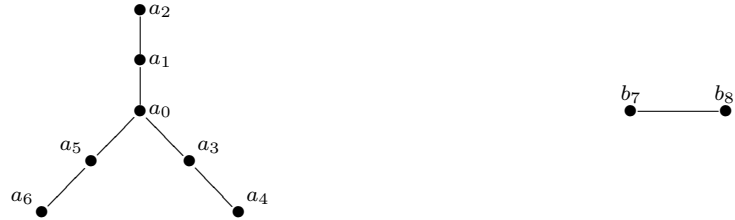
\xymatrix@-1pc{
&& *{\bullet}\\
&& *{\bullet}\ar@{-}[u]_(0.99){a_2}\\
&& *{\bullet}\ar@{-}[u]_(0.01){a_0}_(0.99){a_1} && && && && && *{\bullet}\ar@{-}[rr]^(0.01){b_7}^(0.99){b_8} && *{\bullet}\\
&*{\bullet}\ar@{-}[ur]^(0.01){a_5} & &*{\bullet}\ar@{-}[ul]_(0.01){a_3}\\
*{\bullet}\ar@{-}[ur]^(0.01){a_6} && &&*{\bullet}\ar@{-}[ul]_(0.01){a_4}
}$$

Let $\alpha$ and $\beta$ be the standard coset geometries for $A$ and $B$, respectively. More precisely, the maximal parabolic subgroups of $\alpha$ are $(A_i)_{i \in \{0,1,2,3,4,5,6\}}$ where $A_i=\langle a_j\mid j\in \{0,\ldots,6\}\setminus\{i\}\rangle$ is the subgroup of $A$ generated by all $a_j$'s except for $a_i$. The maximal parabolic subgroups $(B_i)_{i \in \{7,8\}}$ of $\beta$ are defined in the same way.

We let $B$ act on $A$ in the following way: $b_7$ swaps $a_1$ with $a_3$ and $a_2$ with $a_4$, while fixing the remaining generators and $b_8$ swaps $a_3$ with $a_5$ and $a_4$ with $a_6$, while fixing the remaining generators.
This actions clearly permutes the maximal parabolic subgroups $A_i$ of $A$.
For example, as $A_1 =\langle a_0,a_2,a_3,a_4,a_5,a_6\rangle$, we have that $(A_1)^{b_7} = \langle a_0, a_1, a_2, a_4, a_5, a_6 \rangle = A_3$.
The orbits of the induced action of $I_\alpha = \{0,1,2,3,4,5,6\}$ are $\{0\}, \{1,3,5\}$, and $\{2,4,6\}$. Choosing $0,1$ and $2$ as representatives for each orbit, the sets $^LO_M$ are straightforward to compute. For example, we have $^{\{1,3,5\}}O_{\emptyset} =\ ^{\{1,3,5\}}O_{\{8\}} = \{1\}$, $^{\{1,3,5\}}O_{\{7\}} = \{1,3\}$, and $^{\{1,3,5\}}O_{\{7,8\}} = \{1,3,5\}$. It can then be verified that the $(IPO)$ condition holds in this example. 
\end{example}

Suppose now that we have two coset incidence systems $\alpha=(A,(A_i)_{i\in I_\alpha})$ and $\beta=(B,(B_i)_{i\in I_\beta})$ and an action of $B$ on $A$ such that $\alpha$ is $(\beta,\varphi)$-admissible. Let $I = I_\beta \sqcup K$, where $K$ is the set of orbits of the action of $B$ on $I_\alpha$. We will define a new geometry on the type set $I$ for the group $G = A \rtimes_\varphi B$. Let us set $^LO^J := \ ^LO_{I_\beta \setminus J}$ for all $J \subseteq I_\beta$. Notice that $^LO^J$ is the orbit of $F_L$ under the action of $B_J$.

\begin{definition}\label{def:gen_twist}
    Suppose that $\alpha=(A,(A_i)_{i\in I_\alpha})$ and $\beta=(B,(B_i)_{i\in I_\beta})$ are coset incidence systems and that $B$ acts on $A$ such that $\alpha$ is $(\beta,\varphi)$-admissible. Then, the \textit{twisting} of $\alpha$ by $\beta$, with respect to $\varphi$ and a choice of representative $\{F_L\}_{L \in K}$ for the orbits, is the coset incidence system $\TT(\alpha,\beta) = (G,(G_i)_{i \in I})$, where
    the maximal parabolic subgroups $(G_i)_{i \in I}$ are:

\begin{equation}
 G_i= \begin{cases}
     A_{(\cup_{L\in K}(L \setminus ^LO^i))} \rtimes B_i, & \text{if $i \in I_\beta$}.\\
     A_i \rtimes B, & \text{if $i \in K$}.
  \end{cases}
\end{equation}
\end{definition}

Notice that when $i \in K$, we have that $i$ is a subset of element of $I_\alpha$. Therefore the subgroup $A_i$ appearing in the formula $G_i = A_i \rtimes B$ is not necessarily a maximal parabolic subgroup of $\alpha$.

\begin{continuance}{example1}
Coming back to Example~\ref{example1}, we see that $K = \{\{0\}, \{1,3,5\},\{2,4,6\}\}$. Therefore, the twisting $\TT(\alpha,\beta)= (G,(G_i)_{i \in I})$ is a coset incidence system of rank $5$ with $I = \{\{0\},\{1,3,5\},\{2,4,6\},7,8\}$. The maximal parabolic subgroups of $\TT(\alpha,\beta)$ are $G_{\{0\}} = A_{\{0\}} \rtimes B$, $G_{\{1,3,5\}} = A_{\{1,3,5\}} \rtimes B$, $G_{\{2,4,6\}} = A_{\{2,4,6\}} \rtimes B$, $G_7 = A_{\{3,4,5,6\}} \rtimes B_7$, and $G_8 = A_{\{5,6\}} \rtimes B_8$. We can also see on this example that the choice of representatives $\{F_L\}_{L \in K}$ for the orbits is relevant. Indeed, if we had chosen $4$ instead of $2$ as the representative for the orbit $\{2,4,6\}$, the maximal parabolic subgroup $G_{7}$ would now be the group $A_{\{2,3,5\}} \rtimes B_7$ instead of $A_{\{3,4,5,6\}} \rtimes B_7$, two non-isomorphic groups.

\end{continuance}

We now want to show that the twisting is a well-behaved construction. To do so, we start with a technical lemma that will help make the following proofs more readable. The proof of this lemma only uses basic set theoretical arguments, mostly about complements, in the world where $(IPO)$ holds, and is thus omitted.

\begin{lemma}\label{lem:GenSetIntersections}
Given that $\alpha$ in $(\beta, \varphi)$-admissible, the following equalities always hold.
    \begin{enumerate}
        \item $^LO^M \cap  {^LO^N} = {^LO^{M \cup N}}$ for every $L \in K$.
        \item for $L\in K$, $(L \setminus\ ^LO^M) \cup (L \setminus \ ^LO^N) = L \setminus \ ^LO^{M \cup N}$
        \item for $L_1,L_2\in K$ and $M,N\in I_\beta$, $(L_1 \setminus \ ^{L_1}O^M) \cup (L_2 \setminus\ ^{L_2} O^N) = (L_1\cup L_2) \setminus (^{L_1}O^{M}\cup\ ^{L_2}O^N)$
        \item for $L_1,L_2\in K$ and $j,k\in I_\beta$, we have 
        $I_\alpha\setminus( ^{L_1}O^j \cup {}^{L_2}O^j) \cup I_\alpha\setminus( ^{L_1}O^k \cup {}^{L_2}O^k) = I_\alpha\setminus ({}^{L_1}O^{\{j,k\}}\cup {}^{L_2}O^{\{j,k\}} )$
        \item for $L_1\in K$ and $J\subseteq I_\beta$, we have $[I_\alpha\setminus (\cup_{L\in K} {}^LO^J)] \cup L_1 = I_\alpha\setminus (\cup_{L\in K\setminus L_1} {}^LO^J)$
    \end{enumerate}
\end{lemma}

According to Lemma \ref{lem:GenSetIntersections} $(c)$ and $(d)$, we can simplify our maximal parabolic subgroups' description for $i \in I_\beta$ as
$G_i = A_{I_\alpha\setminus(\cup_{L\in K}\ ^LO^i)}\rtimes B_i$.
Therefore, for all $i \in I$, we have $$G_i = A_{I_\alpha\setminus(\cup_{L\in K}\ ^LO^{I_\beta\cap \{i\}})\cup (K\cap \{i\})} \rtimes B_{I_\beta \cap \{i\}}.$$ This unified formula for the maximal parabolic subgroups of $\TT(\alpha,\beta)$ is useful for treating all cases simultaneously in proofs. Notice that the $(\beta, \varphi)$-admissibility of $\alpha$ is essential for this.

\begin{lemma}\label{lem:GenTwistG_J}
    Let $J \subseteq I$ and let $J_\beta = J \cap I_\beta$ and $J_\alpha = J \cap K$. Then, $$G_J = A_{(I_\alpha\setminus (\cup_{L\in K} \ ^LO^{J_\beta})) \cup J_\alpha}\rtimes B_{J_\beta}.$$
\end{lemma}
\begin{proof}
    By definition, we have that $G_J = \cap_{j \in J} G_j$.
    As a set of elements, $G_J$ is then given by $( (\cap_{j \in J_\alpha} A_j) \cap ( \cap_{j \in J_\beta} A_{I_\alpha \setminus (\cup_{L\in K} {}^LO^{j})})) , \cap_{j \in J_\beta} B_j)$. The second factor is equal to $B_{J_\beta}$ as desired. We therefore focus on the first factor. By the definition of parabolical subgroups, we have $\cap_{j \in J_\alpha} A_j = A_{J_\alpha}$ and $\cap_{j \in J_\beta} A_{I_\alpha\setminus(\cup_{L\in K} {}^LO^{\{j\}})} = A_{\cup_{j \in J_\beta}[I_\alpha\setminus(\cup_{L\in K} {}^LO^{\{j\}})]}$. By Lemma \ref{lem:GenSetIntersections}(d), we get that $\cup_{j \in J_\beta} [I_\alpha \setminus (\cup_{L\in K}{}^LO^{\{j\}})] = I_\alpha \setminus (\cup_{L\in K}{}^LO^{J_\beta})$, concluding the proof.
\end{proof}

For all $J \subseteq I$, we define $\JJ = J_\alpha \cup (I_\alpha \setminus (\cup_{L \in K} {^LO^{J_\beta}}))$ so that $G_J = A_{\JJ} \rtimes B_{J_\beta}$. This not only significantly simplifies notations, but also makes them similar to the more intuitive case of direct products. For $i \in I$, we will also use the same notation, so that $G_i = A_{[i_\alpha]} \rtimes B_{i_\beta}$ where $i_\beta = \{i\} \cap I_\beta$. We are know ready to investigate the properties of the twisting geometry. 

\begin{prop}\label{prop:GenTwistingFT}
    If $A$ is flag-transitive on $\alpha$ and $B$ is flag-transitive on $\beta$, then $G$ is flag-transitive on $\TT(\alpha,\beta)$.
\end{prop}
\begin{proof}
    Fix a pair $i,k\in I$ of distinct types, a subset $J \subseteq I\setminus\{i,k\}$, and an element $x\in G_J$. By Proposition~\ref{prop:FTequivs}$(\ref{prop:FTequivs:itm:FTPassini})$, we need to prove that if $G_i\cap xG_k\neq \emptyset$, then $G_J\cap G_i\cap xG_k\neq \emptyset$.
    We have $G_i = A_{[i_\alpha]} \rtimes B_{i_\beta}$, $G_k = A_{[k_\alpha]} \rtimes B_{k_\beta}$ and $G_J = A_{\JJ}\rtimes B_{J_\beta}$ where $i_\beta = \{i\} \cap I_\beta$, $k_\beta = \{k\} \cap I_\beta$, $J_\beta = J \cap I_\beta$ and $J_\alpha = J \cap K$.

    As $x\in G_J$ and $G_i\cap xG_k\neq \emptyset$, we
    can find $y\in G_i$ and $z\in G_k$ such that $y = xz$.
    By the definition of $G_i$, $G_k$ and $G_J$, we have that $y=(y_A,y_B)$, $z = (z_A,z_B)$, and $x = (x_A,x_B)$, with $y_A \in A_{[i_\alpha]},z_A \in A_{[k_\alpha]},x_A\in A_{\JJ}$ and $y_B \in B_{i_\beta},z_B \in B_{k_\beta},x_B\in B_{J_\beta}$.
    We then have that $$y=xz \Leftrightarrow (y_A,y_B)=(x_A,x_B)(z_A,z_B) = (x_A z_A^{x_B},x_B z_B),$$
    where we recall that $ z_A^{x_B} = \varphi(x_B)(z_A)$. 
    The equation $y = xz$ therefore holds if and only if we have $y_A = x_A z_A^{x_B}$ and $y_B = x_B z_B$.
    
    First, note that $y_B = x_B z_B$ implies that $B_{i_\beta} \cap x_B B_{k_\beta}\neq \emptyset$ (recall that $B_\emptyset = B$).
    By the flag-transitivity of $B$ on $\beta$, we then have 
    $B_{J_\beta}\cap B_{i_\beta} \cap x_B B_{k_\beta} \neq \emptyset$. Here, notice that $i$ or $k$ could be equal to an element of $K$, and therefore the fact that $B_{J_\beta}\cap B_{i_\beta}\cap x_B B_{k_\beta} \neq \emptyset$ is not a direct consequence of Proposition~\ref{prop:FTequivs}$(\ref{prop:FTequivs:itm:FTPassini})$ for $\beta$. Instead, this is a consequence of Proposition~\ref{prop:FTequivs}$(\ref{prop:FTequivs:itm:FTIntersectionCosets})$. Indeed, recall that $x_B \in B_{J_\beta}$ so that all the pairwise intersections are not empty.
    The fact that $B_{J_\beta}\cap B_{i_\beta} \cap x_B B_{k_\beta} \neq \emptyset$ in any of the cases then means that we can find
    $t_B\in B_{J_\beta}$, $y'_B\in B_{i_\beta}$ and $z'_B\in B_{k_\beta}$ such that $t_B = y'_B = x_B z_B'$. 

    Similarly, since $y_A = x_A z_A^{x_B}$, we have that 
    $$A_{[i_\alpha]}\cap x_A (A_{[k_\alpha]})^{x_B} \neq \emptyset.$$
    Note that, by the $(\beta,\varphi)$-admissibility of $\alpha$, we have that $B$ acts transitively on the maximal parabolic subgroups of $\alpha$. Hence, by Lemma~2.13 of~\cite{piedade2025group}, $(A_{[k]})^{x_B}$ is a standard parabolic subgroup of $\alpha$.
    Therefore, by Proposition~\ref{prop:FTequivs}$(\ref{prop:FTequivs:itm:FTIntersectionCosets})$, we get that
    $$A_{\JJ} \cap A_{[i_\alpha]}\cap x_A ( A_{[k_\alpha]} )^{x_B} \neq \emptyset.$$
    Once again, it is readily checked that all the pairwise intersections are non-empty.
    We thus find $t_A\in A_{\JJ}$, $y'_A\in A_{[i_\alpha]}$ and $z_A'\in A_{[k_\alpha]}$ such that $t_A = y_A' = x_A (z_A')^{x_B}$.

    Putting both parts together, we have that $(t_A,t_B) = (y_A',y_B') = (x_A (z_A')^{x_B}, x_B z_B')$.
    Clearly $(t_A,t_B)\in G_J$, $(y_A',y_B') \in G_i$ and $(x_A (z_A')^{x_B}, x_B z_B') = (x_A,x_B)(z_A',z_B') \in x G_k$.
    Hence, we get that $G_J\cap G_i\cap x G_k\neq \emptyset$, concluding the proof.
\end{proof}

\begin{prop} \label{prop:GenTwistingRC1}
    If both $\alpha$ and $\beta$ satisfy (RC1), then $\TT(\alpha,\beta)$ also satisfies (RC1).
\end{prop}
\begin{proof}
    Let $J \subset I$ with $|I \setminus J| \geq 2$ and let $J_\beta = J \cap I_\beta$ and $J_\alpha = J \cap K$. We need to show that 
    \begin{equation}\label{eq:RC1Twisting}
          G_J = \langle G_{J \cup \{i\}}\mid i \in I \setminus J \rangle.
    \end{equation}
   By Lemma~\ref{lem:GenTwistG_J}, we have $G_J = A_{[J_\alpha]} \rtimes B_{J_\beta}$, $G_{J \cup \{i\}} = A_{[(J \cup \{i\})_\alpha]} \rtimes B_{(J\cup \{i\})_\beta}$. Moreover, if $i \in K$, we have $G_{J \cup \{i\}} = A_{[J_\alpha] \cup \{i\}} \rtimes B_{J_\beta}$.
   We remind the reader that $A_{I_\alpha}$ and $B_{I_\beta}$ are, respectively, the Borel subgroups of $\alpha$ and $\beta$.

    Suppose first that $K \subseteq J$, so that $J_\alpha = K$. Then, $A_{\JJ} = A_K = A_{I_\alpha}$ and $G_J = A_{I_\alpha} \rtimes B_{J_\beta}$. In this case, Equation~\ref{eq:RC1Twisting} holds because it holds in $\beta$ by assumption and that all parabolic subgroups contain $A_{I_\alpha}$. Indeed, by Lemma~\ref{lem:RCequivs}$(\ref{lem:RCequivs:itm:RC2})$, we get that Equation~\ref{eq:RC1Twisting} is equivalent to $G_J = \langle A_{I_\alpha} \rtimes B_{J_\beta \cup \{i\}}\ , A_{I_\alpha} \rtimes B_{J_\beta \cup \{j\}} \rangle$, for any $i,j \in I_\beta \setminus J_\beta$.
 
 Suppose now that $I_\beta \subseteq J$, so that $J_\beta = I_\beta$. Then, $G_J = A_{[J_\alpha]} \rtimes B_{I_\beta}$. Take any distinct $i,j \in K \setminus J_\alpha$. 
 We get that $G_{J \cup \{i\}} = A_{[J_\alpha] \cup \{i\}} \rtimes B_{I_\beta}$ and $G_{J \cup \{j\}} = A_{[J_\alpha] \cup \{j\}}\rtimes B_{I_\beta}$.
 We thus need to verify that  
 $$\langle A_{[J_\alpha] \cup \{i\}}, A_{[J_\alpha] \cup \{j\}} \rangle = A_{[J_\alpha]}.$$ 
This follows directly from Lemma \ref{lem:RCequivs}$(\ref{lem:RCequivs:itm:Parabolics})$ since $[J_\alpha] = ([J_\alpha] \cup \{i\}) \cap ([J_\alpha] \cup \{j\})$.

Finally, suppose that $J$ contains neither $I_\beta$ nor $K$. Therefore, it is always possible to find $i \in I_\beta \setminus J_\beta$ and $j \in K \setminus J_\alpha$. Hence, we have that $G_{J \cup \{i\}} = A_{[(J \cup \{i\})_\alpha]}\rtimes B_{J_\beta \cup \{i\}}$ and $G_{J \cup \{j\}} = A_{[J_\alpha] \cup \{j\}}\rtimes B_{J_\beta}$. Since $G_J = A_{[J_\alpha]}\rtimes B_{J_\beta}$, we need to show that $B_{J_\beta} \subset \langle G_{J \cup \{i\}}, G_{J \cup \{j\}} \rangle$ and $A_{[J_\alpha]} \subset \langle G_{J \cup \{i\}}, G_{J \cup \{j\}} \rangle$. The first containment holds since $B_{J_\beta}$ is already contained in $G_{J \cup \{j\}}$. For the second one, it is not so direct since the orbits $^L O^{J_\beta \cup \{i\}}$ might be smaller than the orbits $^L O ^{J_\beta}$. However, we can use the $B_{J_\beta}$ factor of $G_{J \cup \{j\}}$  to recover the orbits $^L O ^{J_\beta}$ from the orbits $^L O^{J_\beta \cup \{i\}}$. Therefore, using the action of $B_{J_\beta}$, we can recover $A_{[J_\alpha]}$ from $A_{[(J \cup \{i\})_\alpha]}$, which shows that $A_{[J_\alpha]} \subset \langle G_{J \cup \{i\}}, G_{J \cup \{j\}} \rangle$.

\end{proof}

\begin{prop} \label{prop:GenTwistingFIRM}
    If both $\alpha$ and $\beta$ satisfy (FIRM), then $\TT(\alpha,\beta)$ also satisfies (FIRM).
\end{prop}
\begin{proof}
    Since $\alpha$ and $\beta$ satisfy (FIRM), we have $A_{I_\alpha\setminus \{i\}} \neq A_{I_\alpha}$ and $B_{I_\beta\setminus\{j\}}\neq B_{I_\beta}$, for each $i\in I_\alpha$ and $j\in I_\beta$.
    Moreover, by Lemma~\ref{lem:GenTwistG_J} together with Lemma \ref{lem:GenSetIntersections}(e), we obtain that $G_{I}=A_{I_\alpha}\rtimes B_{I_\beta}$ and
    \begin{equation*}
        G_{I\setminus\{i\}} =
        \begin{cases}
            A_{I_\alpha}\rtimes B_{I_\beta\setminus\{i\}} &, \textnormal{ if } i\in I_\beta \\
            A_{(I_\alpha\setminus {^i} O ^{I_\beta})}\rtimes B_{I_\beta}&, \textnormal{ if } i\in K
        \end{cases}
    \end{equation*}

    Note that $^i O ^{I_\beta}$ is never empty as it always contains $F_i$. Therefore, $I_\alpha \setminus ^i O ^{I_\beta}$ is never equal to $I_\alpha$ and thus $A_{(I_\alpha\setminus {^i} O ^{I_\beta})} \neq A_{I_\alpha}$.  
    
    As $\alpha$ and $\beta$ satisfy (FIRM), we get that $G_{I}\neq G_{I\setminus\{i\}}$, for each $i\in I$, as desired.
\end{proof}

We group all of the above results, and a little bit more, in the following theorem.

\begin{thm}\label{thm:GenTwisting_FT_Geo_RC_FRM}
    Let $\alpha=(A,(A_i)_{i\in I_\alpha})$ and $\beta=(B,(B_i)_{i\in I_\beta})$ be two coset incidence systems, $\varphi: B\rightarrow Aut(A)$ a group homomorphism, and let $\alpha$ be $(\beta,\varphi)$-admissible. Suppose that $A$ and $B$ are flag-transitive on $\alpha$ and $\beta$, respectively. Then
    \begin{enumerate}
        \item the group $G=A\rtimes_\varphi B$ is flag-transitive on $\TT(\alpha,\beta)$;
        \item $\TT(\alpha,\beta)$ is a coset geometry;
        \item $\TT(\alpha,\beta)$ is a finite coset geometry if and only if $\alpha$ and $\beta$ are both finite coset geometries;
        \item $\TT(\alpha,\beta)$ is residually connected if both $\alpha$ and $\beta$ are residually connected;
        \item $\TT(\alpha,\beta)$ is firm if both $\alpha$ and $\beta$ are firm. Moreover, if the Borel subgroups $A_{I_\alpha}$ and $B_{I_\beta}$ of $\alpha$ and $\beta$ are trivial, the twisting $\TT(\alpha,\beta)$ is thin if both $\alpha$ and $\beta$ are thin;
    \end{enumerate}
\end{thm}
\begin{proof}
    Part (a), (b), (c) and (d) are direct consequences of Proposition~\ref{prop:GenTwistingFT}, Theorem~\ref{thm:cosetFT}, Theorem~\ref{thm:CosetFirmThinFinite} and ~\ref{prop:GenTwistingRC1}.

    It remains thus only to show (e). We already know that $\TT(\alpha,\beta)$ is flag-transitive. Since both $\alpha$ and $\beta$ are supposed to be firm, by Proposition~\ref{prop:GenTwistingFIRM} and Theorem~\ref{thm:CosetFirmThinFinite}, we have that $\TT(\alpha,\beta)$ is at least firm.
    
    Suppose now that $\alpha$ and $\beta$ are thin and that $B_{I_\beta}$ and $A_{I_\alpha}$ are trivial. This implies that $|A_{I_\alpha\setminus\{i\}} \backslash A_{I_\alpha}|=2$ and $|B_{I_\beta\setminus\{j\}} \backslash B_{I_\beta}|=2$, for each $i\in I_\alpha$ and $j\in I_\beta$ and that, for $L\in K$, we have ${}^LO^{I_\beta}=\{F_L\}$, the representative of the orbit $L$.
    Therefore, we get that 
    \begin{equation*}
        G_{I\setminus\{i\}} =
        \begin{cases}
            B_{I_\beta\setminus\{i\}}&, \textnormal{ if } i\in I_\beta \\
            A_{I_\alpha \setminus \{F_i\}}&, \textnormal{ if } i\in K 
        \end{cases}
    \end{equation*}
    and $G_{I}=\{1\}$.
    This shows that the minimal parabolic subgroups of $\TT(\alpha,\beta)$ are always minimal parabolic subgroups of $\alpha$ or $\beta$.
    Hence, we get that $|G_{I\setminus\{i\}} \backslash G_{I}|=2$, showing that $\TT(\alpha,\beta)$ is thin.
\end{proof}
Note that the condition on the Borel subgroups in the last part of the above theorem is always satisfied in practice. Indeed, the Borel subgroup of any thin coset geometry is always a normal subgroup and we can obtain an isomorphic geometry by taking the quotient by the Borel (see~\cite{Tits1957}). More explicitly, suppose that $\beta = (B,(B_i)_{i \in I_\beta})$ is thin. It is then always isomorphic to the coset geometry $(B/B_{I_\beta}, (B_i/B_{I_\beta})_{i \in I\beta})$ whose Borel subgroup is now trivial. It is therefore natural to consider that all thin coset geometries have trivial Borel subgroup.

 As a direct corollary, we obtain that the twisting of two hypertopes, and so in particular of two polytopes, is always a hypertope.
\begin{coro}\label{coro:TwistingHypertope}
    Let $\alpha$ and $\beta$ are regular hypertopes. 
    For a group homomorphism $\varphi: B\rightarrow Aut(A)$, if $\alpha$ is $(\beta,\varphi)$-admissible, then $\TT(\alpha,\beta)$ is a regular hypertope.
\end{coro}
\begin{proof}
    This is a direct consequence of Theorem~\ref{thm:GenTwisting_FT_Geo_RC_FRM} and the definition of a regular hypertope.
\end{proof}

\begin{continuance}{example1}
    In Example~\ref{example1}, for any choice of representatives $\{F_L\}_{L\in K}$, the twisting $\TT(\alpha,\beta)$ is a regular hypertope. Indeed, both $\alpha$ and $\beta$ are regular hypertopes and $\alpha$ is $(\beta,\varphi)$-admissible for any choice of representatives $\{F_L\}_{L\in K}$.
    Even though the group $G = A\rtimes_\varphi B$ is isomorphic for any choice, this choice of representative yields different $C$-groups and different geometrical structures.

    If the set of representatives is chosen to be $\{0,1,2\}$, the minimal parabolic subgroups are $G^{\{0\}} = \langle a_0\rangle$, $G^{\{1,3,5\}} = \langle a_1\rangle$, $G^{\{2,4,6\}} = \langle a_2\rangle$, $G^7=\langle b_7\rangle$ and $G^8 = \langle b_8\rangle$.
    Note that, by the residual connectedness of $\TT(\alpha,\beta)$, we can construct the maximal parabolic subgroups using only $a_0, a_1, a_2, b_7,$ and $b_8$ as generators. The Coxeter diagram of the resulting $C$-group $(G,\{a_0,a_1,a_2,b_7,b_8\})$ is the following. 
    $$
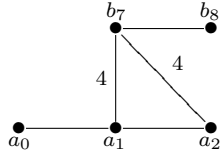
\xymatrix@-1pc{
    && *{\bullet}\ar@{-}[rr]^(0.01){b_7}^(0.99){b_8}\ar@{-}[dd]_4 \ar@{-}[ddrr]^4 && *{\bullet}\\
    \\
    *{\bullet}\ar@{-}[rr]_(0.01){a_0}_(0.99){a_1} && *{\bullet}\ar@{-}[rr]_(0.99){a_2} && *{\bullet}
    }$$
    Indeed, as $(a_1a_3)^2 = e$ and $a_3 = a_1^{b_7}$, we get$(a_1b_7a_1b_7)^2 = (a_1b_7)^4 = e$. Note that there are more relations in this group beside those shown by the Coxeter diagram.

    Instead, if the set of representatives is chosen to be $\{0,1,6\}$, the automorphism group of $\TT(\alpha,\beta)$ is the $C$-group $(G,\{a_0,a_1,a_6,b_7,b_8\})$ with the following Coxeter diagram.
    $$
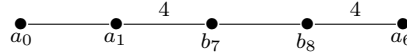
\xymatrix@-1pc{
    *{\bullet}\ar@{-}[rr]_(0.01){a_0}_(0.99){a_1} && *{\bullet}\ar@{-}[rr]_(0.99){b_7}^4  && *{\bullet}\ar@{-}[rr]_(0.99){b_8} && *{\bullet}\ar@{-}[rr]^4_(0.99){a_6} && *{\bullet}
    }$$
    In this case, $\TT(\alpha,\beta)$ is an abstract regular polytope, since its diagram is linear.
\end{continuance}

We conclude this section by an application of the twisting construction to a self-dual regular tetrahedron. The general case of self-dual polytopes will be explored in Section \ref{sec:applications}. This example shows that it is possible to give more geometrical meaning to the twisting, even though the general case is quite involved and thus omitted.

\begin{example}
    Let $\alpha$ be the incidence geometry of the tetrahedron. The automorphism group of $\Gamma$ is $A=Sym(4)$ and $\alpha = (Sym(4), (\langle \rho_1, \rho_2 \rangle, \langle \rho_0, \rho_2 \rangle, \langle \rho_0,\rho_1 \rangle ))$ where $\rho_0 = (1,2)$, $\rho_1 = (2,3)$ and $\rho_2 = (3,4)$. The permutation $\tau = (1,4)(2,3)$ acts by conjugation on $Sym(4)$. This action fixes $\rho_1$ and exchanges $\rho_0$ and $\rho_2$. We are thus in the setting where $\beta = (B =\langle \tau \rangle, (\{e\}))$ is a rank $1$ incidence system with an action $\varphi$ of $B =\langle \tau \rangle$ on $A =Sym(4)$. The coset geometry $\alpha$ is easily checked to be $(\beta, \varphi)$-admissible. We obtain a coset geometry $\TT(\alpha,\beta)$ whose automorphism group is $G = Sym(4) \rtimes C_2$. Notice that $(A_0)^\tau = A_2$ and $(A_1)^\tau = A_1$. For notational purposes, we consider that the type set of $\TT(\alpha,\beta)$ is $I = \{ \{0,2\}, \{1\}, \tau\}$ and we choose $0$ and $1$ as representatives for the orbits. The maximal parabolical subgroups are $G_{\{0,2\}} = \langle \rho_1 \rangle \rtimes \langle \tau \rangle$, $G_{\{1\}} = \langle \rho_0,\rho_2\rangle \rtimes \langle \tau \rangle$ and $G_\tau = \langle \rho_0, \rho_1 \rangle $. By Corollary \ref{coro:TwistingHypertope}, we know that $\TT(\alpha,\beta)$ is a regular hypertope of rank $3$. In fact, $\TT(\alpha,\beta)$ is the geometry of the cube. In this case, the twisting can be interpreted geometrically. Indeed, cosets of $G_{\{1\}}$ in $G$ correspond to cosets $\langle \rho_0, \rho_2 \rangle$ in $Sym(4)$ which in turn correspond to edges of the tetrahedron $\Gamma$. Similarly, cosets of $G_{\{0,2\}}$ in $G$ correspond to cosets $\langle \rho_1 \rangle$ in $Sym(4)$ which correspond to flags of type $\{0,2\}$ of the tetrahedron $\Gamma$. The cosets of $G_\tau$ in $G$ come in two flavours. They are either of the form $(g \cdot \langle \rho_0,\rho_1 \rangle,e)$ or $(g \cdot \langle \rho_1,\rho_2 \rangle, \tau)$. Hence, half of the cosets of $G_\tau$ in $G$ correspond to vertices of the tetrahedron $\Gamma$ and the other half correspond to faces of $\Gamma$. The coset geometry $\TT(\alpha,\beta)$ can thus be constructed explicitly as an incidence geometry of rank three over the type set $\{0,1,2\}$ as follows:
    \begin{itemize}
        \item Elements of type $0$ are the vertices and the faces of the tetrahedron $\Gamma$,
        \item Elements of type $1$ are given by pairs $\{v,F\}$ where $v$ is a vertex of $\Gamma$ and $F$ is a face of $\Gamma$ that contains $v$,
        \item Elements of type $2$ are the edges of $\Gamma$,
        \item Incidence is naturally induced by the incidence in the tetrahedron $\Gamma$. In other words, if $X$ is the set of elements and $x,y \in X$, we have that $x$ is incident to $y$ if and only if $x \cup y $ is a flag of $\Gamma$.
    \end{itemize}
    The isomorphism between the geometry of the cube and the one described here can be seen visually by inscribing the tetrahedron into a cube, as depicted in Figure \ref{fig:SimplexInCube}. In the picture, it becomes apparent that half the vertices of the cube correspond to vertices of the tetrahedron and that the remaining vertices correspond to faces of the tetrahedron. Also, every face of the cube contains exactly one edge of the tetrahedron. Finally, every edge of the cube contains two vertices, one of which correspond to a vertex of the tetrahedron and the other correspond to a face. It is then clear that the bijection we just described from the elements of the cube to the one of the tetrahedron preserves incidences.

\end{example}

\begin{figure}[h!]\centering
    \begin{tikzpicture}
        \draw (0,0) -- (4,0);
        \draw (4,4) -- (4,0);
        \draw (0,4) -- (4,4);
        \draw (0,0) -- (0,4);
        \draw (0,4) -- (3,6);
        \draw (0+4,4) -- (3+4,6);
        \draw (0+4,0) -- (3+4,6-4);
        \draw (3+4,6-4) -- (3+4,6);
        \draw (3,6) -- (3+4,6);
        \draw (3,6) --(3,2) -- (0,0);
        \draw (3,2) -- (7,2);


        \draw[dotted] (0,0) -- (4,4) -- (7,2) -- (0,0);
        \draw[dotted] (0,0) -- (3,6) -- (4,4);
        \draw[dotted] (3,6) -- (7,2);

        \filldraw[olive](0,0)circle (2pt);    
        \filldraw[olive](4,4)circle (2pt);
        \filldraw[olive](7,2)circle (2pt);
        \filldraw[olive](3,6)circle (2pt);  
        \filldraw[teal](4,0)circle (2pt);
        \filldraw[teal](0,4)circle (2pt);
        \filldraw[teal](3,2)circle (2pt);
        \filldraw[teal](7,6)circle (2pt);    
    \end{tikzpicture}
    \caption{An explicit representation of the twisting of a self-dual regular tetrahedron}
    \label{fig:SimplexInCube}   
\end{figure}
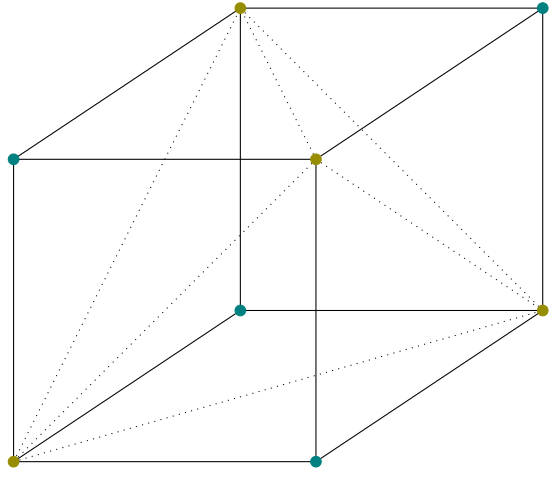

\section{Wreath Product}\label{sec:Wreath}

Let $A$ be a group, $\Omega$ be a finite set, and $B$ a group acting on the left on $\Omega$. Let $\prod_\Omega A$ be the direct product of copies of $A$ indexed by $\Omega$. In other words, $\prod_\Omega A$ is the set of elements $\bar{a}=(^\omega a)_{\omega \in \Omega}$ with $^\omega a\in A$ for all $\omega \in \Omega$. Note that we use superscripts before the letters for indexes in $\Omega$ instead of the usual under-scripts/superscripts after the letters, as the latter are reserved for parabolical subgroups. 

The action of $B$ on the left of $\Omega$ can be extended to an action on $\prod_\Omega A$ by setting $b (^\omega a)_{\omega \in\Omega} := (^{b^{-1}(\omega )}a)_{\omega \in\Omega}$.
The wreath product $ A \wr_\Omega B$ is then the semidirect product $(\prod_\Omega A) \rtimes B$, where the action of $B$ on $\prod_\Omega A$ is the one described above.

Let $\alpha=(A,(A_i)_{i\in I_\alpha})$ be a coset incidence system. As we have seen in Section \ref{sec:DirectProduct}, we can construct the product coset incidence system $\overline\alpha := \prod_{\omega \in \Omega} \alpha$ over $I_{\overline\alpha} =\sqcup_{\omega \in \Omega} I_\alpha$. The type set $I_{\overline{\alpha}}$ is a disjoint union of copies of $I_\alpha$ indexed by $\Omega$. To clarify notation, for each $\lambda \in \Omega$ we will denote by $I_\alpha^\lambda$ the copy of $I_\alpha$ indexed by $\lambda$ in $I_{\overline{\alpha}}$. Using the generalized twisting defined in Section \ref{sec:GenralizedTwisting}, we can then construct a geometry for $A \wr_\Omega B$, as long as the action of $\prod_{\omega \in \Omega} \alpha$ is $(\beta,\varphi)$-admissible.

\begin{lemma}
    The direct product $\prod_{\omega \in \Omega} \alpha$ is $(\beta,\varphi)$-admissible if and only if the orbits of the action of $B$ on $\Omega$ satisfy $(IPO)$.
\end{lemma}
\begin{proof}
    The condition $(a)$ of $(\beta,\varphi)$-admissibility is always satisfied. Indeed, let $i \in I_{\overline\alpha} = \sqcup_{\omega \in \Omega} I_\alpha$ and let $\lambda$ be the unique element of $\Omega$ such that $i \in I_\alpha^\lambda$. The parabolic subgroup $(\prod_{\omega \in \Omega} A)_i$ of $\prod_{\omega \in \Omega} \alpha$ is the direct product of copies of $A$ except that the $\lambda$ factor is replaced by $A_i$. Noticing that $I_{\overline{\alpha}} \cong I_\alpha \times \omega$, we can write $i = (j,\lambda)$ with $j \in I_\alpha$ and denote $(\prod_{\omega \in \Omega} A)_i$ by $G_{(j,\lambda)}$. The action of an element $b \in B$ then sends $G_{(j,\lambda)}$ to the parabolical subgroup $G_{(j,b \cdot \omega)}$. This also directly implies that the orbits of the action of $B$ on $I_{\overline{\alpha}}$ are in bijection with $\mathcal{O} \times I_\alpha$ where $\mathcal{O}$ is the set of orbits of the action of $B$ on $\Omega$. Two orbits $(O_1,j)$ and $(O_2,k)$ with $j \neq k \in I_\alpha$ are always disjoint. It is then clear that $(IPO)$ holds for an orbit $(O,j)$ under the action of $B$ if and only it holds for the orbits $O$ under the same action.
\end{proof}

\begin{definition}\label{def:wreath}
    Let $\alpha= (A, (A_i)_{i \in I_\alpha})$ and $\beta = (B,(B_i)_{i \in I_\beta})$ be two coset incidence systems. Let $\varphi$ be an action of $B$ on a finite index set $\Omega$ such that the orbits of $\varphi$ on $\Omega$ satisfy $(IPO)$. Let $\overline{\alpha} = \prod_{\omega \in \Omega} \alpha$ be the product coset geometry over the type set $I_{\overline{\alpha}} = \sqcup_{\omega \in \Omega} I_\alpha ^\omega$. Then, the wreath product of $\alpha$ by $\beta$, with respect to $\varphi$ and to a choice of representative for the orbits, is the coset incidence system $\alpha \wr_\Omega \beta = \TT(\overline{\alpha}, \beta)$ over the type set $I = K \sqcup I_\beta$ where $K$ is the set of orbits of $\varphi$ on $I_{\overline{\alpha}}$.
\end{definition}

For clarity, we make explicit a formula for the maximal parabolic subgroups of the wreath product incidence system.
\begin{prop}
Let $i \in I$. Then,
$$G_i = \begin{cases}
     (\prod_{\omega \in (\cup_{L \in K_\Omega} L \setminus ^L O^i)} {}^\omega A ) \rtimes B_{i}, & \text{if $i \in I_\beta$}.\\
     \overline{A_{i}} \rtimes B, & \text{if $i \in K$}.
  \end{cases}$$
  where $\overline{A_i} = \bigcap_{(j,\omega)\in i} \overline{A_{(j,\omega)}}$ and $K_\Omega$ is the set of orbits of $B$ on $\Omega$.
\end{prop}
\begin{proof}
    This follows directly by Definition~\ref{def:gen_twist} of the maximal parabolic subgroups of the twisting. 
\end{proof}

Finally, we group all of the properties of the wreath product of incidence systems.

\begin{thm}\label{thm:wreath}
    Let $\alpha=(A,(A_i)_{i\in I_\alpha})$ and $\beta=(B,(B_i)_{i\in I_\beta})$ be two coset incidence systems and suppose that $B$ acts on some set $\Omega$ in such a way that this action satisfies $(IPO)$. Suppose also that $A$ and $B$ act flag-transitively on $\alpha$ and $\beta$ respectively. Then
    \begin{enumerate}
        \item The group $A \wr_\Omega B$ acts flag-transitively on $\alpha \wr_\Omega \beta$;
        \item The coset incidence system $\alpha \wr_\Omega \beta$ is a geometry;
        \item $\alpha \wr_\Omega \beta$ is residually connected if $\alpha$ and $\beta$ residually connected;
        \item $\alpha \wr_\Omega \beta$ is firm (resp. thin) if both $\alpha$ and $\beta$ are firm (resp. thin).
    \end{enumerate}
\end{thm}
\begin{proof}
    This follows from Theorem~\ref{thm:GenTwisting_FT_Geo_RC_FRM} and Theorem~\ref{thm:AlphaTimesBetaInfinite}.
\end{proof}

\section{Applications}\label{sec:applications}

In this section, we show a few applications of the tools developed earlier in the article.

\subsection{High rank string $C$-groups represented by permutation representation graphs}

There is currently an active research on polytopes whose automorphism group $G$ is a transitive proper subgroup of $\textnormal{Sym}(n)$. Particularly, there is a strong interest for groups of high rank, meaning that the number $r$ of generating involutions for $G$ satisfies $r\geq n/2$. 
In this line of research, one can obtain a classification of string groups generated by involutions (groups generated by an ordered set of involutions such that any pair of non-consecutive involutions commutes) with the desired properties. However, in order to obtain regular polytopes, one must then show that these groups are actually string $C$-groups. In order to do so, the intersection property given in Section~\ref{subsec:StringCGroups} must be verified. This final step is usually difficult (see ~\cite{FernandesPiedade2026,CameronFernandesLeemans2016TransitiveSubsSn}).

In this section, we will show how the operations defined in this article and in~\cite{piedade2025group} allow to verify that the intersection property holds for many of these groups.
In particular, we will pick some cases from~\cite{FernandesPiedade2026}, for which the intersection property was not verified yet.
We will start with the easy case of a direct product, and move on to more complicated cases.

\subsubsection{Direct product}

Let $G$ be the permutation group defined by the following permutation representation graph (graph (1) of Table 9 in \cite{FernandesPiedade2026}).

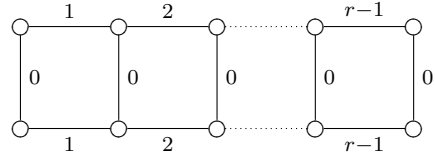
$$\xymatrix@-1pc{
        *+[o][F]{}\ar@{-}[dd]^0\ar@{-}[rr]^1&&*+[o][F]{}\ar@{-}[dd]^0\ar@{-}[rr]^2&&*+[o][F]{}\ar@{-}[dd]^0\ar@{.}[rr] &&*+[o][F]{}\ar@{-}[dd]^0\ar@{-}[rr]^{r-1}&& *+[o][F]{}\ar@{-}[dd]^0  \\
        &&&&&&&&\\
        *+[o][F]{}\ar@{-}[rr]_1&&*+[o][F]{}\ar@{-}[rr]_2&&*+[o][F]{}\ar@{.}[rr]&&*+[o][F]{} \ar@{-}[rr]_{r-1}&& *+[o][F]{} }$$
We see that $G=\langle \rho_0,\rho_1,\ldots,\rho_{r-1}\rangle$ with
$\rho_0=(1,2)(3,4)(5,6)\ldots(2r-1,2r)$ and that
$\rho_i=(2i-1,2i+1)(2i,2i+2)$, for $1\leq i\leq r-1$. Notice that $\rho_0\rho_i=\rho_i\rho_0$, for $1\leq i\leq r-1$, and that $\rho_0\notin\langle \rho_1,\ldots,\rho_{r-1}\rangle$. Hence the group $G$ splits as a direct product $G=\langle\rho_0\rangle\times\langle\rho_1,\ldots,\rho_{r-1}\rangle$.
Let $A=\langle \rho_0\rangle$ and $B=\langle \rho_1,\ldots,\rho_{r-1}\rangle$, and consider the coset incidence systems $\alpha=(\langle\rho_0\rangle,(\{e\}))$, over $I_\alpha=\{0\}$, and $\beta=(B,(B_i)_{i\in I_\beta})$, over $I_\beta=\{1,\ldots,r-1\}$, where $B_i=\langle \rho_j\mid j\in I_\beta\setminus\{i\}\rangle$. The coset incidence system $\alpha$ is trivially a residually-connected thin coset geometry with $A$ acting flag-transitively on $\alpha$.
On the other hand, $B\cong Sym(r)$ and $\beta$ is the incidence geometry of a $(r-1)$-simplex, a very well known abstract regular polytope, and thus a residually-connected thin coset geometry on which $B$ acts flag-transitively.

Therefore, by Theorem~\ref{thm:AlphaTimesBeta}, we have that $\alpha \times \beta$ is s flag-transitive, residually-connected and thin coset geometry. Therefore, $\alpha \times \beta$ is a regular hypertope and $G$ is thus a $C$-group. In particular, letting $S=\{\rho_0,\ldots,\rho_{r-1}\}$, we get that $(G,S)$ satisfies the intersection property. The geometry $\alpha \times \beta$ is in fact an abstract regular polytope, as all pair of generators $\rho_i,\rho_j\in S$ commute whenever $|i-j|>1$.

Notice that the exact same arguments above can be used to prove that the permutation representation graph (2) of Table 9 of~\cite{FernandesPiedade2026} is also a string $C$-group. Indeed, the only difference is that $\langle \rho_1,\ldots,\rho_{r-1}\rangle$ is a transitive permutation representation of $Sym(r)$ of degree $2r$.

\subsubsection{Wreath product}
Let $G$ be the group defined by the following permutation representation graph (graph (13) of Table 5 of ~\cite{FernandesPiedade2026}).

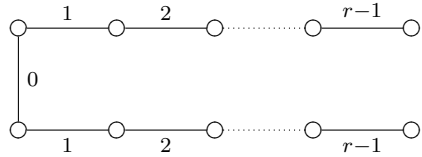
$$\xymatrix@-1pc{
        *+[o][F]{}\ar@{-}[dd]^0\ar@{-}[rr]^1&&*+[o][F]{}\ar@{-}[rr]^2&&*+[o][F]{}\ar@{.}[rr] &&*+[o][F]{} \ar@{-}[rr]^{r-1}&& *+[o][F]{}  \\
        &&&&&&&&\\
        *+[o][F]{}\ar@{-}[rr]_1&&*+[o][F]{}\ar@{-}[rr]_2&&*+[o][F]{}\ar@{.}[rr]&&*+[o][F]{} \ar@{-}[rr]_{r-1}&& *+[o][F]{} }$$

In this case, the group is $G=\langle \rho_0,\rho_1,\ldots,\rho_{r-1}\rangle$ where $\rho_0=(1,2)$ and
$\rho_i=(2i-1,2i+1)(2i,2i+2)$, for $1\leq i\leq r-1$. 
Contrary to the previous case, the first involution $\rho_0$ does not commute with all other generators. Instead, we have  $G\cong C_2\wr_{\Omega} Sym(r)$, where $\Omega=\{1,\ldots,r\}$. We will thus use the wreath product construction (see Definition \ref{def:wreath}) for this case. 

Let $A=\langle \rho_0 \rangle$ with associated coset geometry $\alpha=(A,(\{e\}))$ over $I_\alpha=\{0\}$ and let $B=\langle \rho_1,\ldots,\rho_{r-1}\rangle\cong Sym(r)$ with associated coset geometry $\beta = (B,(B_i)_{i\in I_\beta})$ over $I_\beta=\{1,\ldots,r-1\}$.
We have $G = A\wr_{\Omega} B = (\prod_{\omega\in \Omega}{}^\omega A)\rtimes B$, where $^\omega A = \langle \rho_{(0,\omega)}\rangle$ with $\rho_{(0,\omega)}=(2\omega-1,2\omega)$. Notice that when $\omega=1$, we get that $^1 A=\langle (1,2)\rangle ) =A$.

We will now construct $\overline\alpha :=\prod_{\omega\in \Omega}{}^\omega\alpha$, with $I_{\overline{\alpha}}=\{(0,\omega)\mid \omega\in \Omega\}$, where $^\omega\alpha = (^\omega A,(\{e\}))$ over $I_{^\omega \alpha}=\{(0,\omega)\}$. Let us denote $\overline A=\prod_{\omega\in \Omega}{}^\omega A$, and for $\tau\in\Omega$, $ \overline{A_{(0,\tau)}}=\prod_{\omega\in \Omega\setminus\{\tau\}}{}^\omega A$.
For example, for $r\geq 3$, we have that $\overline{A_{(0,1)}}=\prod_{\omega\in \Omega\setminus\{1\}}{}^\omega A = {}^2A\times{}^3A\times\cdots\times {}^rA= \langle (3,4),(5,6),\ldots,(2r-1,2r)\rangle $. 

We now need to verify that the action of $B$ satisfies the $(IPO)$ condition. Notice that the action of $B\cong Sym(r)$ on $\Omega$ is transitive. This implies that the action of $B$ on the type set $I_{\overline\alpha}$ is also transitive.
Indeed, we have $\rho_i (\overline{A_{(0,i)}})\rho_i=\overline{A_{(0,i+1)}}$, for every $i\in\{1,\ldots,r-1\}\subset \Omega$. Hence, the set of orbits is $K=\{I_{\overline \alpha}\}$. Let $L=I_{\overline \alpha}$ and $F_L=(0,1)$. It is easy to check that, for $M =I_\beta\setminus\{k\}$, with $k\in I_\beta$, we have $^LO_M=\{(0,i)\mid 1\leq i\leq k\}$. In general, for $M\subseteq I_\beta$,
we have $^LO_M=\{(0,i)\mid 1\leq i\leq k+1\}$, where $k$ is the largest integer such that $M'=\{1,2,\ldots,k\}$ and $M'\subseteq M$. If $1\notin M$, then $k=0$ and ${}^LO_M=\{(0,1)\}$.
Hence for $M,N\subseteq I_\beta$, we have that $^LO_M\cap {}^LO_N=\{(0,i)\mid 1\leq i\leq \min(k_M,k_N)+1\}$, where $k_M$ and $k_N$ are again defined as the respective largest integers such that a consecutive sequence of integers starting with $1$ sits inside $M$ and $N$, respectively. As $\min(k_M,k_N)$ is the largest value of a consecutive sequence starting in 1 found in $M\cap N$, we have that $^LO_M\cap {}^LO_N={}^LO_{M\cap N}$. Therefore, the $(IPO)$ condition holds.

We can then build the coset geometry $\alpha\wr_{\Omega}\beta =(G,(G_i)_{i\in I})$ over $I=\{I_{\overline{\alpha}},1,2,\ldots,r-1\}$ where $G=A\wr_\Omega B$, and the maximal parabolic subgroups are
$G_{I_{\overline{\alpha}}}=B=\langle \rho_1,\ldots,\rho_{r-1}\rangle$, $G_i=\overline{A}_{I_{\overline{\alpha}}\setminus\{(0,k)\mid 1\leq k\leq i\}}\rtimes B_i$, for $1\leq i\leq r-1$. Notice that we can rewrite $G_i = \langle (2k-1,2k)\mid 1\leq k\leq i\rangle\rtimes B_i$, for $1\leq i\leq r-1$. Also, we have $(2k-1,2k)=(1,2)^{\rho_1\rho_2\ldots\rho_{k-1}}$. Hence, we obtain that $G_i=\langle\rho_0, \rho_j\mid j\in I_\beta\setminus\{i\}\rangle$ for $i\in I_\beta$. Notice that these maximal parabolics are exactly the maximal parabolics you would obtain by taking all but one of the generators given by the permutation representation graph. Hence, as both $\alpha$ and $\beta$ are regular polytopes, we get that $\alpha\wr_\Omega \beta$ is a regular hypertope and that $G$ is a $C$-group. Finally, to prove that $G$ is a string $C$-group, we can notice that the ordered set of generators $S=\{\rho_0,\ldots,\rho_{r-1}\}$ satisfies the commuting property, by inspecting the permutation representation graph.

The construction above also works for graph (15) of Table 6 of~\cite{FernandesPiedade2026}, with the only difference being that in that case $B\cong Sym(r)$ has a transitive permutation representation on $2r$ points.

Consider now the following permutation group $G$ (graph (14) of Table 5 of ~\cite{FernandesPiedade2026}).
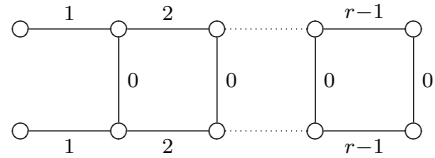
$$\xymatrix@-1pc{
        *+[o][F]{}\ar@{-}[rr]^1&&*+[o][F]{}\ar@{-}[dd]^0\ar@{-}[rr]^2&&*+[o][F]{}\ar@{-}[dd]^0\ar@{.}[rr] &&*+[o][F]{}\ar@{-}[dd]^0\ar@{-}[rr]^{r-1}&& *+[o][F]{}\ar@{-}[dd]^0  \\
        &&&&&&&&\\
        *+[o][F]{}\ar@{-}[rr]_1&&*+[o][F]{}\ar@{-}[rr]_2&&*+[o][F]{}\ar@{.}[rr]&&*+[o][F]{} \ar@{-}[rr]_{r-1}&& *+[o][F]{} }$$
For this group, we have that $\rho_0=(3,4)(5,6)\ldots(2r-1,2r)$ and $\rho_i=(2i-1,2i+1)(2i,2i+2)$, for $1\leq i\leq r-1$.
Notice that the identity of the group $G$ depends on the value of $r$.
If $r$ is even, $\rho_0$ is an odd permutation and $G\cong C_2\wr Sym(r)$. However if $r$ is odd, $\rho_0$ is an even permutation and $G\cong (C_2\wr Sym(r))^+$, the index 2 subgroup of $C_2\wr Sym(r)$ with only even permutations fixing the block system.
When $r$ even, we can show that the group represented by the above permutation representation graph is a string $C$-group by the exact same process as above. The main difference is simply how one defines $^\omega A= \langle \rho_{(0,\omega)}\rangle$, where $\rho_{(0,\omega)} = \lambda(2\omega-1,2\omega)$ and $\lambda = (1,2)(3,4)\ldots(2r-1,2r)$.

Now that we have taken care of the case of $r$ even, we can use that to address the case of $r$ odd also. 
Notice that $G_0 =\langle \rho_1,\ldots,\rho_{r-1}\rangle \cong Sym(r)$ and $G_{r-1} = \langle \rho_0,\rho_1,\ldots,\rho_{r-2}\rangle$ is a sesqui-extension on the first generator of $H$ (see~\cite{Fernandes2012} for more details), where $H$ is the group defined by the same permutation representation graphs but with rank $r-1$. As $r-1$ is even, we have that $H$ is a string $C$-group, hence by Proposition 3.3 of~\cite{Fernandes2012}, we already know that $G_{r-1}$ is a string $C$-group. In order to prove that $G$ is a $C$-group, using Proposition 2E16 of~\cite{ARP}, we only need to prove that $G_0\cap G_{r-1} = \langle \rho_1,\ldots,\rho_{r-2}\rangle \cong Sym(r-1)$. That is easily verified by noticing that $Sym(r-1)$ is a maximal subgroup of $G_0\cong Sym(r)$. Finally, the string condition is easily checked from evaluating the action of the involutions in its  permutation representation graph.

We can also handle graph (16) of Table 5 of~\cite{FernandesPiedade2026}, when $r$ is even, by the same process. However, in this case, the sesqui-extension argument does not hold, so, when $r$ is odd, we cannot guarantee that $G$ is a $C$-group by using our methods.

We summarize the results of this section in the following proposition.
\begin{prop}
    Let $(G,S)$ be a string group generated by involutions. 
    \begin{enumerate}
        \item If $(G,S)$ is  the permutation group given by the permutation graphs (1) or (2) in Table 9 of~\cite{FernandesPiedade2026}, or by the graphs (13), (14) or (15) in Table 5 of~\cite{FernandesPiedade2026}, then $(G,S)$ is a string $C$-group;
        \item If $(G,S)$ is the permutation group given by the permutation graph (16) in Table 5 of~\cite{FernandesPiedade2026} and is of even rank, then $(G,S)$ is a string $C$-group.
    \end{enumerate}
\end{prop}

\subsection{Twisting Correlations}

Correlations are symmetries of an incidence system which permute the types of the elements (see Section \ref{sec:back:subsec:incidence}). Let $\alpha = (A,(A_i)_{i \in I_\alpha})$ be a coset incidence geometry and let $B = \Out(A)$ be the outer automorphism group of $A$. In this setting, the group $B$ naturally comes with an action $\varphi$ on $A$. As long as $B$ can naturally be considered as an coset incidence system $\beta$ such that $\alpha$ is $(\beta,\varphi)$-admissible, we can then construct the twisting $\TT(\alpha,\beta)$. 

This setting can for example be realized in the context of self-dual regular polytopes.

\begin{thm}\label{thm:SelfDualPoly}
    Let $\alpha$ be a regular self-dual polytope of rank $n \geq 3$ with automorphism group $A = \langle \rho_0, \rho_1, \cdots, \rho_{n-1}\rangle$. Then,
    \begin{enumerate}
        \item There exists $\tau \in \Aut(A)$ such that $\tau(\rho_i) = \rho_{n-i-1}$;
        \item There exists a regular hypertope of rank $\lceil n/2 \rceil +1$ with automorphism group $A \rtimes \langle \tau \rangle$. Moreover, if $n > 3$, the hypertope can be chosen such that its diagram is not linear;
        \item If $n = 3,4,5$, there exists a regular polytope of rank $\lceil n/2 \rceil +1$ with automorphism group $A \rtimes \langle \tau \rangle$;
        \item Moreover, if the duality $\tau \in \Aut(A)$ from (a) is not an inner automorphism, the group $A \rtimes \langle \tau \rangle$ is a subgroup of $\Aut(A)$.
        \end{enumerate} 
\end{thm}

\begin{proof}
    (a): Let $C$ be a base chamber of $\alpha$. Since $\alpha$ is regular and admits a duality, there exists a duality $\varphi$ of $\alpha$ that sends the chamber $C$ to itself, while permuting the elements inside of it. Then, there must exist an automorphism $\tau$ of $A$ inducing this duality $\varphi$ (see \cite[Proposition 19]{TrialitySuzuki}, for example).
    
    (b): The hypertope in question is $\Gamma =\TT(\alpha,\beta)$ where $\beta = (\langle \tau \rangle, \{e\})$. Indeed, part (a) guarantees that $\alpha$ is $(\beta,\tau)$-admissible. The type set of $\Gamma$ is $I = \{\{0,n-1\}, \{1,n-2\}, \cdots, \{\lfloor n/2 \rfloor, \lceil n/2 \rceil\}, \{\tau\}\} $ if $n$ is odd and $I = \{\{0,n-1\}, \{1,n-2\}, \cdots, \{n/2\}, \{\tau\}\}$ if $n$ is even. In both cases, the rank is thus $\lceil n/2 \rceil +1$. Moreover, if we choose as orbit representatives the elements $0,1,\cdots, \lfloor n/2 \rfloor$, we can guarantee that $\Gamma$ will have a non-linear diagram as long as $n > 3$. Indeed, the automorphism group of $\Gamma$ will be the $C$-group $\langle \rho_0,\rho_1, \cdots, \rho_{\lfloor n/2 \rfloor}, \tau \rangle = A \rtimes \langle \tau \rangle$. Moreover, the order of $\rho_i \tau$ is equal to two times the order of $\rho_i \rho_{n-i-1}$ if $\rho_i$ is not fixed by $\tau$ and the order of $\rho_i \tau$ is equal to two if $\rho_i$ is fixed by $\tau$. Therefore, the diagram of $\Gamma$ is a triangle when $n = 4$, a square when $n = 5$ and for $n > 5$, we have that the vertex of the diagram corresponding to $\tau$ will be connected to at least $3$ other vertices, which forces the diagram to be non-linear.

    (c): When $n = 3$, part $(b)$ show that $\Gamma$ is always a polytope. When $n = 4$, the same reasoning as in part (b) shows that, if we choose $0$ and $3$ as orbit representatives, instead of $0$ and $1$, we get that $\Gamma$ has a linear diagram, and is thus a polytope. Similarly, if $n= 5$, choose $0 ,2$ and $3$ as orbit representatives again guarantees that $\Gamma$ is a polytope.

    (d): If the duality is not an inner automorphism, we have that $A \rtimes \langle \tau \rangle$ is naturally isomorphic to $\langle A,\tau \rangle \subset \Aut(A)$.
\end{proof}

We now use these results to show the existence of regular polytopes and regular hypertopes for automorphism groups of sporadic simple groups. All computations used where executed using either {\sc Magma}~\cite{magma} or {\sc GAP}~\cite{GAP4}.
\begin{thm}
    Let $H$ be a sporadic simple group, and let $G$ be such that $H < G \leq \Aut(H)$ where the first inclusion is strict. Then there exist rank $3$ regular polytopes and regular hypertopes with non-linear diagrams with automorphism group $G$, except possibly if $H$ is $Fi_{22}$ and $Fi_{24}'$.
\end{thm}

\begin{proof}

Let $H$ be a sporadic simple group and consider $G=\Aut(H)$.
We always have that $G = H$ or $|G:H|=2$. Our hypothesis rule out the case of $G= H$ and so we will always get that $G = \Aut(H)$ and that $|G:H| = 2$.

Assume first that $H \in \{M_{12},J_2,J_3, HS,He,Suz\}$.
For each of these groups, we find a self-dual regular polytope of rank 4 whose automorphism group is $H$ by looking at the atlas in~\cite{Leemans2006}. We then verify computationally that the duality for these polytopes is induced by an outer automorphism. We can then apply Theorem~\ref{thm:SelfDualPoly} to obtain a rank 3 regular polytope and and a rank 3 regular hypertope with non-linear diagram whose automorphism group is $\Aut(H)$. 

Suppose now that $H = ON$ is the O'Nan group. 
For this group, there exists exactly one self-dual rank $4$ regular polytope of type $\{8,4,8\}$~\cite{CONNOR2014} and computations show that its duality is an inner automorphism. We therefore cannot use rank $4$ polytopes in this case.
In the data of the Atlas~\cite{Leemans2006} we can find rank $3$ self-dual regular polytopes with automorphism group the O'Nan group and such that the duality is not an inner automorphism of $\Aut(H)$ (note that the size of the set of rank $3$ regular polytopes for the O'Nan group is 18Gb). In particular, there exists a self-dual rank $3$ regular polytope of type $\{k,k\}$ with $k$ odd. Letting $\rho_0,\rho_1$ and $\rho_2$ be the generators of this polytope and applying Theorem~\ref{thm:SelfDualPoly} to it, we get a rank $3$ regular polytope with automorphism group $G = \Aut(H)=\langle\rho_0,\rho_1\rangle\rtimes\langle \tau\rangle$. Note that, as $\tau$ fixes $\rho_1$ and $\tau\rho_0\tau=\rho_2$, $o(\tau\rho_1)=2$ and $o(\tau\rho_0)=4$. Hence, the resulting polytope is of type $\{4,k\}$, with $k$ an odd integer $\geq 3$.
To obtain a rank $3$ regular hypertope, we then apply the halving operation on the generators $\rho_1$ and $\rho_0$ (for more details, see~\cite{halving,MonteroWeiss}).
Indeed, the group $\langle \tau, \rho_0, \rho_1\rho_0\rho_1\rangle$ is isomorphic to $\Aut(H)$ (notice that $\rho_1 = (\rho_0\rho_1\rho_0\rho_1)^{(k-1)/2}\rho_0$). However, as $o(\tau\rho_1\rho_0\rho_1) = 4$, we have that $\Aut(H)$ is a quotient of the triangle group $(4,k,4)$, which does not have a linear diagram. Computationally, we can check that $\langle \tau, \rho_0, \rho_1\rho_0\rho_1\rangle$ is the automorphism group of a rank $3$ regular hypertope with non-linear diagram.

Suppose now that $H = M_{22}$. Contrary to the above cases, there are no polytopes for the simple group $M_{22}$ (see~\cite{Mazurov2003,Leemans2025}).
In the atlas in~\cite{Leemans2006}, we can find regular polytopes of rank $3$ for $\Aut(M_{22})$. To get rank $3$ regular hypertopes with non-linear diagram, we can again apply the halving operation~\cite{halving} to one of these rank $3$ regular polytope. For example, consider the regular polytope $\{4,12\}$ with automorphism group $\langle \rho_0,\rho_1,\rho_2\rangle\cong \Aut(M_{22})$, where
\begin{align*} 
\rho_0 &= (3, 15)(4, 16)(5, 9)(6, 19)(8, 18)(10, 17)(14, 20)\\
\rho_1 &= (1, 3)(2, 19)(4, 5)(6, 22)(7, 18)(8, 12)(9, 16)(10, 21)(11, 17)(13, 15)(14,20)\\
\rho_2 &= (1, 7)(2, 12)(5, 8)(6, 20)(9, 18)(11, 22)(13, 21)(14, 19).
\end{align*}
We apply the halving to this polytope to obtain a regular hypertope $\langle \rho_0, \rho_1, \rho_1^{\rho_2}\rangle$. We then confirm computationally that the automorphism group of this hypertope is indeed $\Aut(M_{22})$ and that its diagram is not linear.

Finally, suppose that $H = McL$. Again, there are no polytopes for the simple group $McL$ (see~\cite{Mazurov2003,Leemans2025}). However, contrary to the case of $M_{22}$, we do not have any library of regular polytopes for $\Aut(McL)$.
Hence, we have computationally found a rank $3$ polytope for $\Aut(McL)$ and then applied the halving operation to it in order to get a regular hypertope.
We show here the pseudocode used in order to obtain the rank $3$ polytope. Let $G = \Aut(McL)$,   $t\in G$ such that $t^2=1$ and $t\notin McL$. Note that ${\verb|IP(K)|}$ is a function that checks the intersection property ofr the maximal parabolic subgroups $\verb| <t, inv_0>|$, $\verb| <t, inv_1>|$ and $\verb| <inv_0, inv_1>|$.
\begin{verbatim}
Invs := Involutions(G);
for inv_1 in Invs
  if Order(inv_1 * t) = 2 then   ## inv_1 commutes with t 
    for inv_0 in Invs
      if Order(inv_0 * t) >= 3 and Order(inv_0*inv_1) is odd then
        K = Group<t, inv_0, inv_1>
        if #K = #G and IP(K) = true then
            print "Group<t, inv_0, inv_1> is Aut(H)"
            break all loops
        end if
      end if
    end for
  end if
end for    
\end{verbatim}
From the code above, we can get an abstract regular polytope of Schl\"afli type $\{6,5\}$
with $\Aut(McL)=\langle t, inv_0, inv_1\rangle$ as its automorphism group. Then, we have that $\langle t, inv_0, inv_0^{inv_1}\rangle$ is still isomorphic to $\Aut(McL)$, with the coset incidence system $(G,(\langle t,inv_0\rangle,\langle t,inv_0^{inv_1}\rangle,\langle inv_0, inv_0^{inv_1}\rangle)$ being a rank $3$ regular hypertope with non-linear diagram.

\end{proof}

We are currently computing regular polytopes of rank $3$ with automorphism groups the sporadic Fisher groups $Fi_{22}$ and $Fi_{24}'$, but since these groups are extremely large, it will take some time.

\subsection{Lamplighter Group}

We conclude this paper with an interesting application that is slightly outside of the scope of the tools we developed. Indeed, we will apply the twisting of coset incidence systems in the context of an infinite set $\Omega$. This shows simultaneously that the twisting and wreath product construction for coset incidence systems can be extended to infinite index sets, but that one must be careful in doing so.
Notice that when $\Omega$ is infinite, the wreath product $A \wr_\Omega B$ of two groups can be defined as $(\oplus_\Omega A) \rtimes B$ or $(\prod_\Omega A) \rtimes B$, and those are 
different. The second one is usually called the unrestricted wreath product of $A$ by $B$.

The lamplighter group $G$ is a classical example of wreath product. Its name comes from its natural action on $(\oplus_\mathbb{Z} C_2)$ that can be interpreted as an infinite street with lampposts that can be turned on or off. As such, $G$ is the (restricted) wreath product $ C_2 \wr \mathbb{Z}$ where $\mathbb{Z}$ acts by shifts on $(\oplus_\mathbb{Z} C_2)$. It is well know that $G = \langle a , t \mid a^2, (at^nat^{-n}), \forall n\in \mathbb{N}\rangle$ and that $G$ is not finitely presentable. 
Let $\alpha = (C_2, (\{e\}))$ and $\beta = (\mathbb{Z}, (\{e\}))$ be the two rank one coset geometries associated to the group $C_2$ and $\mathbb{Z}$.
As shown in Section~\ref{sec:DirectProduct}, if we take the infinite sum $\bar{\alpha} = \bigoplus_\mathbb{Z} \alpha$, the group $\oplus_\mathbb{Z}C_2$ does not act flag-transitively on $\bar{\alpha}$, as there exists a pair of flags of the same type that differ on infinitely many elements. Nonetheless, the group $\mathbb{Z}$ acts on $\alpha$ and we have that $\alpha$ is $(\beta,\varphi)$-admissible, where $\beta = (\mathbb{Z}, \{e\})$. Therefore, we can consider $\Gamma =\TT(\alpha,\beta)$ in this context. Notice that $\TT(\alpha,\beta) = (G,(C_2,\mathbb{Z}))$. Hence, $\Gamma$ is a rank two coset incidence geometry for the lampligher group. Notice that $\Gamma$ is flag-transitive, as rank two coset incidence geometries are always flag-transitive. Moreover, $\Gamma$ is also connected. Indeed, $\Gamma$ can be described explicitly as an incidence geometry of rank two over the type set $\{0,1\}$. Its elements of type $0$ correspond to states of the street, meaning that they are elements of $(\oplus_\mathbb{Z} C_2)$. Elements of type $1$ are street states together with exactly one uncertainty, meaning that the state of every lamp is known except for one. An uncertain state is then incident to the both states where the uncertainty is resolved, by either turning off or on the unknown lamp state, and to nothing else. It is clear from this description that $\Gamma$ is connected. We will call $\Gamma$ the \textit{Schr\" odinger street geometry}. We emphasize that $A = \oplus_\mathbb{Z}C_2$ does not act flag-transitively on $\alpha$ but that the lampligher group $\oplus_\mathbb{Z}C_2 \rtimes \mathbb{Z}$ still ends up acting flag-transitively on $\Gamma$.

What happens if we build $\bar{\alpha}$ from the direct product $\prod_\mathbb{Z}C_2$ instead? As before, let $\alpha = (C_2, (\{e\}))$ and $\beta = (\mathbb{Z}, (\{e\}))$ be the two rank one coset geometries associated to the group $C_2$ and $\mathbb{Z}$. Let $\bar{\alpha} = \prod_\mathbb{Z}\alpha$. This time, $\prod_\mathbb{Z} C_2$ does act flag-transitively on $\bar{\alpha}$ and we can again construct $\overline{\Gamma} = \TT(\alpha,\beta)$. We have $\overline{\Gamma} = (\prod_\mathbb{Z}C_2 \rtimes \mathbb{Z},(C_2,\mathbb{Z}))$. Therefore, we notice that $\overline{\Gamma}$ is not connected. Indeed, if $\overline{\Gamma}$ were to be connected, it would imply that the unrestricted lamplighter group $\prod_\mathbb{Z}C_2 \rtimes \mathbb{Z}$ would be generated by two of its elements, and in particular that it would be finitely generated, which is not the case. This shows that Proposition \ref{prop:GenTwistingRC1} does not hold in general when we allow $I_\alpha$ to be infinite. Indeed, the problem in this case is that $\prod_\mathbb{Z}C_2$ is not generated by the union of its factors.

\bibliographystyle{ieeetr} 
\bibliography{refs}

\end{document}